\documentclass[reqno, 11pt]{amsart}
\usepackage{palatino,euler}
\usepackage[foot]{amsaddr}
\usepackage{amsfonts}
\usepackage{amsthm,thmtools,thm-restate,enumerate}
\usepackage[margin=1in]{geometry}
\usepackage[colorlinks,
    linkcolor={red},
    citecolor={blue},
    urlcolor={blue}]{hyperref}
\usepackage{graphicx}
\usepackage[all,cmtip]{xy}

\newtheorem{lemma}{Lemma}[section]
\newtheorem{theorem}[lemma]{Theorem}

\newtheorem{proposition}[lemma]{Proposition}
\newtheorem{conjecture}[lemma]{Conjecture}

\theoremstyle{definition}
\newtheorem{definition}[lemma]{Definition}
\newtheorem{remark}[lemma]{Remark}

\usepackage{mathtools,tikz, tikz-cd, subcaption}
\usetikzlibrary{arrows,decorations.pathmorphing,backgrounds,calc,math,intersections,angles}
\tikzstyle{map}=[->,semithick]
\tikzstyle{arc}=[bend left,->,semithick]
\tikzstyle{rinclusion}=[right hook->,semithick]
\tikzstyle{linclusion}=[left hook->,semithick]

\def\Ri{\mathcal R}

\def\R{{\mathbb R}}

\def\K{{\mathcal K}}

\def\eps{{\varepsilon}}

\def\G{{\mathcal{G}}}
\def\sh{{\mathscr{Sh}}}

\def\h{{\mathrm{H}}}

\def\A{{\mathcal A}}

\newcommand{\diam}[2][\M]{\mathrm{diam}_{#1}\left(#2\right)}

\def\diam{\operatorname{diam}}

\def\conv{\operatorname{Conv}}

\begin{document}

\title{The Shadow of Vietoris--Rips Complexes in Limits}

\author{Kazuhiro Kawamura}
\address{Department of Mathematics, University of Tsukuba, Japan}
\email{kawamura@math.tsukuba.ac.jp}

\author{Sushovan Majhi}
\address{Data Science Program, George Washington University, USA}
\email{s.majhi@gwu.edu}

\author{Atish Mitra}
\address{Mathematics Department, Montana Technological University, USA}
\email{amitra@mtech.edu}

\begin{abstract} 
The Vietoris--Rips complex, denoted $\Ri_\beta(X)$, of a metric space $(X,d)$ at scale $\beta$ is an abstract simplicial complex where each $k$-simplex corresponds to $(k+1)$ points of $X$ within diameter $\beta$.
For any abstract simplicial complex $\K$ with the vertex set $\K^{(0)}$ a Euclidean subset, its shadow, denoted $\sh(\K)$, is the union of the convex hulls of simplices of $\K$.
This article centers on the homotopy properties of the shadow of Vietoris--Rips complexes $\K=\Ri_\beta(X)$ with vertices from $\mathbb{R}^N$, along with the canonical projection map $ p\colon \Ri_\beta (X) \to \sh(\Ri_\beta(X))$.
The study of the geometric/topological behavior of $p$ is a natural yet non-trivial problem. 
The map $p$ may have many ``singularities'', which have been partially resolved only in low dimensions $N\leq 3$.
The obstacle naturally leads us to study systems of these complexes $\{ \sh(\Ri_{\beta}(S)) \mid \beta > 0, S\subset X\}$. We address the challenge posed by singularities in the shadow projection map by studying systems of the shadow complex using inverse system techniques from shape theory, showing that the limit map exhibits favorable homotopy-theoretic properties. More specifically, leveraging ideas and frameworks from Shape Theory, we show that in the limit ``$\beta \to 0$ and $S \to X$'', the limit map ``$\lim p$'' behaves well with respect to homotopy/homology groups when $X$ is an ANR (Absolute Neighborhood Retract) and admits a metric that satisfies some regularity conditions.  This results in limit theorems concerning the homotopy properties of systems of these complexes as the proximity scale parameter approaches zero and the sample set approaches the underlying space (e.g., a submanifold or Euclidean graph).
The paper concludes by discussing the potential of these results for finite reconstruction problems in one-dimensional submanifolds.
\end{abstract}

% Questions:
% 1. What is the exact relation between ANR and M1-M3; What happens when M is a manifold?
% 2. Define ANR at some point
% 3. Should we simplify the shadow notation?
% 4. Should we need consitent notation for direct limit?
% 5. Discuss the convention for $\iota_{\beta,\gamma}^{S}$
% 6. It might be better to omit $\infty$, but let us decide it later.
% 7. Please let us decide which $i_{\beta,\gamma}$ or $i_{\gamma,\beta}$ to choose


%\pacs[MSC Classification noicesess2020]{55P10 (Primary), 55N31, 54E35 (Secondary)}

\maketitle

\section{Introduction}
\begin{definition}[The Vietoris--Rips Complex]\label{def:VR}
Given a metric space $(X,d_X)$ and a positive proximity scale $\beta$, the \emph{Vietoris--Rips} complex of $X$ at scale $\beta$, denoted $\Ri_\beta(X)$, is defined to be an abstract simplicial complex having an $m$--simplex for every finite subset of $\sigma\subset\A$ with cardinality $(m+1)$ and diameter less than $\beta$.
More concretely,
\[
\Ri_{\beta}(X) = \{ \sigma\mid \sigma ~\mbox{is a finite subset of }~ A, \diam_{d_X}(\sigma) < \beta \}.
\]    
\end{definition}
The strict inequality in the above definition is essential to this paper.  
For simplicity, the geometric realization of $\Ri_\beta(X)$ endowed with the Whitehead topology \cite{spanier1995algebraic} is also denoted by the same symbol.

\begin{figure}[tbh]
\centering
\includegraphics[scale=0.1]{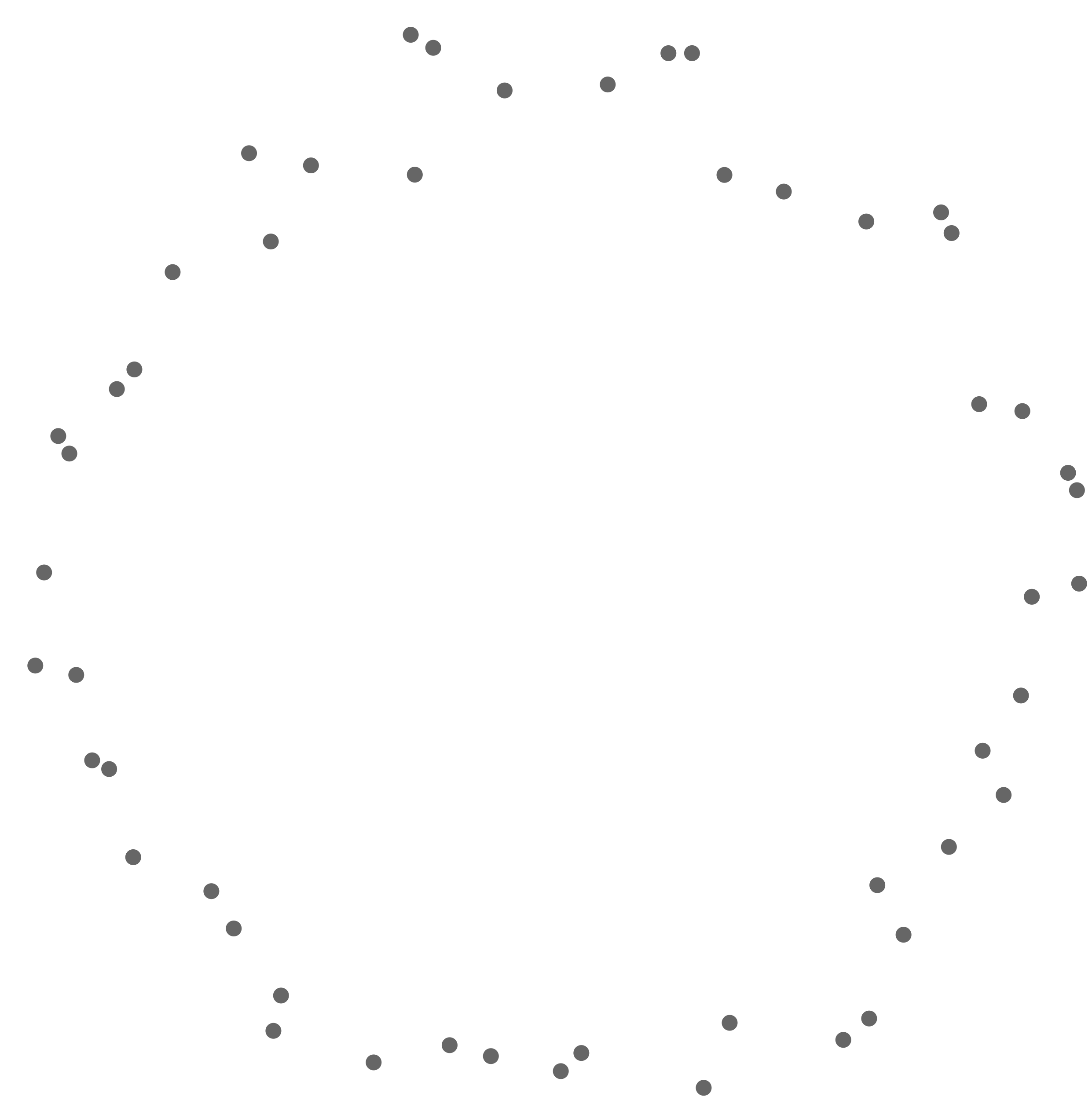}
\quad
\includegraphics[scale=0.1]{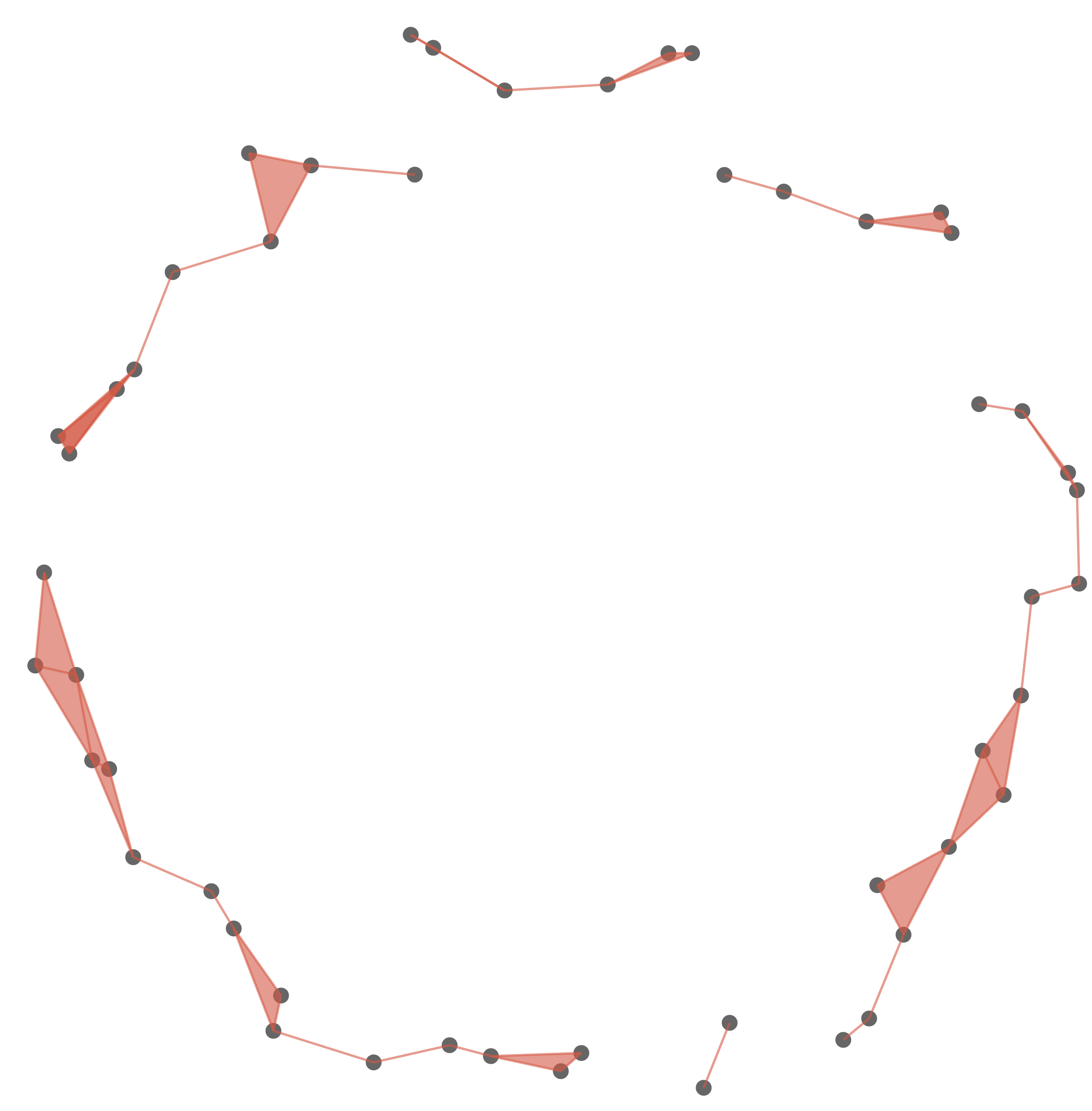}
\quad
\includegraphics[scale=0.1]{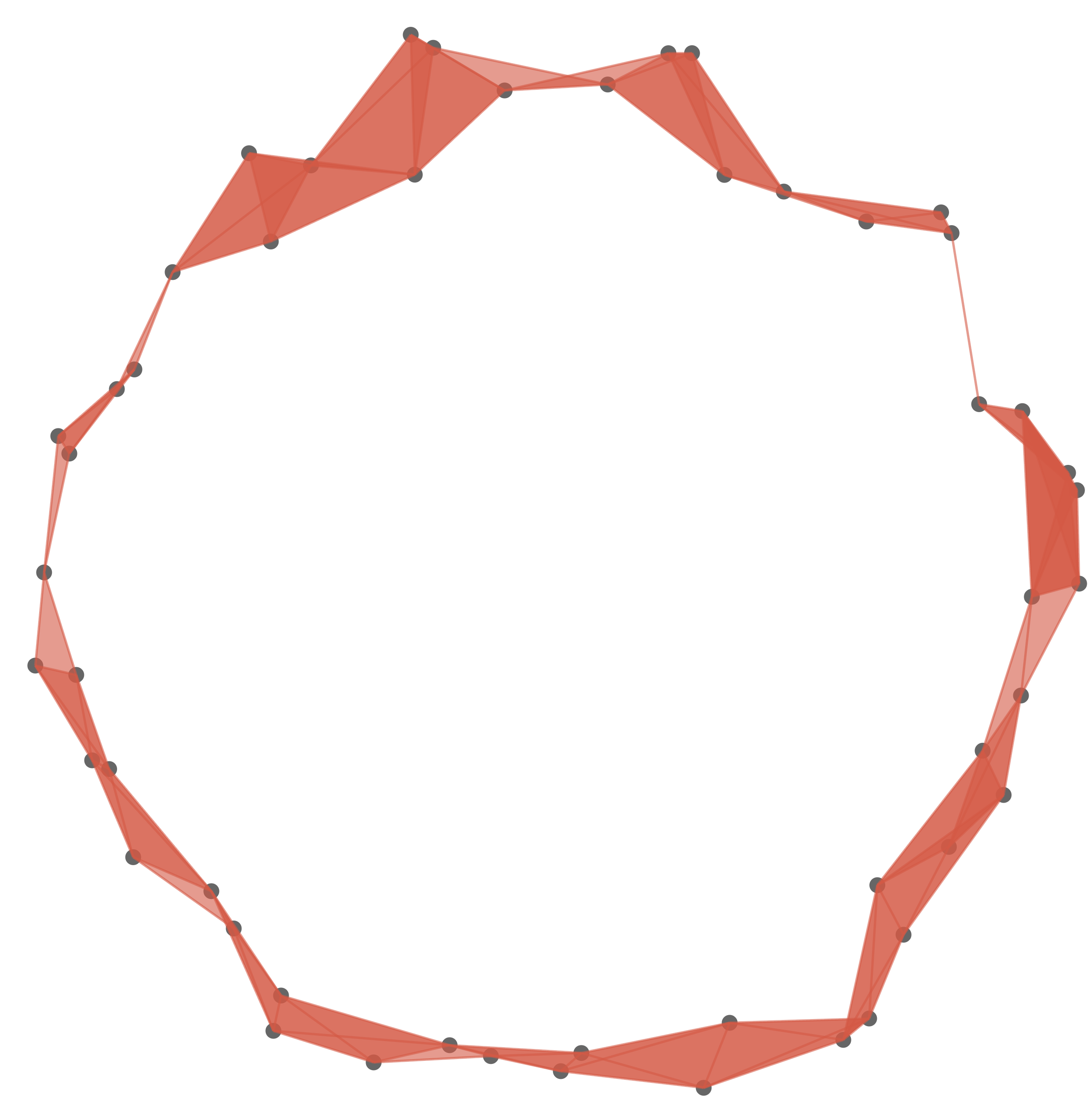}
\caption{Vietoris--Rips complexes on a point-cloud for a growing (left-to-right) scale $\beta$. As $\beta$ grows, the topology of the complex becomes more and more connected until it eventually becomes contractible.}
\label{fig:rips}
\end{figure}

The concept was initially introduced by L. Vietoris in 1927~\cite {Vietoris1927} and subsequently studied extensively by E. Rips, particularly in the context of hyperbolic groups. 
Despite its early 20th-century inception, it has only been within the last decade that these complexes have gained increasing popularity, especially within the applied topology and topological data analysis (TDA) communities. 
The computational simplicity of Vietoris–Rips complexes makes them a more palatable choice for applications compared to traditional alternatives like the \v{C}ech complexes and $\alpha$-complexes~\cite{ghrist2014elementary,edelsbrunner2022computational}. 

This combinatorial flexibility, however, is balanced by a theoretical cost: the topology of the Vietoris–Rips complex of a metric space---even a finite one---is generally poorly understood. Nonetheless, there have been noteworthy developments in the study of Vietoris–Rips complexes constructed for near Riemannian manifolds~\cite{hausmann1995vietoris,latschev2001vietoris,majhi2024demystifying}, metric graphs~\cite{majhi2023vietoris,KomendarczykMajhiMitra2025-shadow}, and general geodesic spaces of bounded Alexandrov curvature~\cite{MajhiStability}.

Hausmann's pioneering work established that any closed Riemannian manifold \(M\) is homotopy equivalent to its Vietoris–Rips complex \(\mathcal{R}_{\beta }(M)\) for sufficiently small scales \(\beta \) \cite{hausmann1995vietoris}. This fundamental result naturally motivated the \emph{finite reconstruction problem}: identifying the conditions under which $M$ remains homotopy equivalent to the Vietoris–Rips complex of a finite, dense sample.
Latschev in~\cite{latschev2001vietoris} addressed this problem by extending the reconstruction context to metric spaces close to $M$ in the Gromov--Hausdorff sense ~\cite{burago2022course}. 
Latschev’s Theorem states:
\emph{For a closed Riemannian manifold \(M\), there exists a constant \(\epsilon _{0}(M)>0\) such that for any scale \(0<\beta \le \epsilon _{0}(M)\), there exists a \(\delta (\beta )>0\) where any metric space \(S\) satisfying \(d_{\mathrm{GH}}(S,M)<\delta (\beta )\) yields a Vietoris--Rips complex $\mathcal{R}_{\beta}(S)$ homotopy equivalent to $M$}.
While this result highlights that the sampling threshold $\epsilon _{0}$ depends strictly on the intrinsic geometry of $M$, it remains purely qualitative and existential. 
More recently, the author of \cite{majhi2024demystifying} provided a quantitative and practical analogue of Latschev’s result for manifolds, which was subsequently extended to more general metric spaces with curvature bounds in~\cite{MajhiStability}.

\begin{figure}[htb]
\centering
\begin{subfigure}[t]{0.3\textwidth}
\centering
\begin{tikzpicture}[scale=1.2]
\path[name path=1] (3, 1.5) -- (0,1);
\filldraw[thick, fill=blue!10, draw=blue!50, name path=2] (0, 0) -- (2,0) -- (1,2) -- cycle;
\path[name path=3] (2.5, 2) -- (1.2,0.4);
\draw[thick, draw=blue!50, name intersections={of=1 and 2}] (3, 1.5) -- (intersection-1);
\draw[thick, draw=blue!50, name intersections={of=1 and 2}, dashed] (intersection-1) -- (intersection-2);
\draw[thick, draw=blue!50, name intersections={of=1 and 2}] (intersection-2) -- (0,1);
\draw[dashed, thick, draw=blue!50, name intersections={of=2 and 3}] (1.2, 0.4) -- (intersection-1);
\draw[thick, draw=blue!50, name intersections={of=2 and 3, name=E}, name intersections={of=1 and 3, name=F},shorten >=0.1cm] (E-1) -- (F-1);
\draw[thick, draw=blue!50, name intersections={of=1 and 3, name=F}, shorten <=0.1cm] (F-1) -- (2.5, 2);
\filldraw[fill=blue!10, draw=blue!80] (0, 0) circle (2pt);
\filldraw[fill=blue!10, draw=blue!80] (2, 0) circle (2pt);
\filldraw[fill=blue!10, draw=blue!80] (1, 2) circle (2pt);
\filldraw[fill=blue!10, draw=blue!80] (0, 1) circle (2pt);
\filldraw[fill=blue!10, draw=blue!80] (3, 1.5) circle (2pt);
\filldraw[fill=blue!10, draw=blue!80] (2.5, 2) circle (2pt);
\filldraw[fill=blue!10, draw=blue!80] (1.2, 0.4) circle (2pt);
\end{tikzpicture}
\end{subfigure}%
~ 
\begin{subfigure}[t]{0.3\textwidth}
\centering
\begin{tikzpicture}[scale=1.2]
\path[name path=1] (3, 1.5) -- (0,1);
\filldraw[thick, fill=blue!10, draw= blue!30, name path=2] (0, 0) -- (2,0) -- (1,2) -- cycle;
\path[name path=3] (2.5, 2) -- (1.2,0.4);
\draw[thick, draw=blue!30, name intersections={of=1 and 2}] (3, 1.5) -- (intersection-1);
\draw[thick, draw=blue!30, name intersections={of=1 and 2}] (intersection-2) -- (0,1);
\draw[thick, draw=blue!30, name intersections={of=2 and 3, name=E}, name intersections={of=1 and 3, name=F},shorten >=0.1cm] (E-1) -- (2.5,2);
\end{tikzpicture}
\end{subfigure}%
~
\begin{subfigure}[t]{0.3\textwidth}
\centering
\begin{tikzpicture}[scale=1.2]
\path[name path=1] (3, 1.5) -- (0,1);
\filldraw[thick, fill=blue!10, draw=blue!50, name path=2] (0, 0) -- (2,0) -- (1,2) -- cycle;
\path[name path=3] (2.5, 2) -- (1.2,0.4);
\draw[thick, draw=blue!50, name intersections={of=1 and 2}] (3, 1.5) -- (intersection-1);
\draw[thick, draw=blue!50, name intersections={of=1 and 2}] (intersection-1) -- (intersection-2);
\draw[thick, draw=blue!50, name intersections={of=1 and 2}] (intersection-2) -- (0,1);
\draw[thick, draw=blue!50, name intersections={of=2 and 3}] (1.2, 0.4) -- (intersection-1);
\draw[thick, draw=blue!50, name intersections={of=2 and 3, name=E}, name intersections={of=1 and 3, name=F},shorten >=0.1cm] (E-1) -- (F-1);
\draw[thick, draw=blue!50, name intersections={of=1 and 3, name=F}, shorten <=0.1cm] (F-1) -- (2.5, 2);
\draw[thick, draw=blue!50, name intersections={of=1 and 2, name=E}, name intersections={of=1 and 3, name=F}] (E-2) -- (1.2, 0.4);
\draw[thick, draw=blue!50, name intersections={of=1 and 2, name=E}, name intersections={of=1 and 3, name=F}] (E-1) -- (1.2, 0.4);
\draw[thick, draw=blue!50] (1.2, 0.4) -- (0,0);
\draw[thick, draw=blue!50] (1.2, 0.4) -- (2,0);
\filldraw[fill=blue!10, draw=blue!80] (0, 0) circle (2pt);
\filldraw[fill=blue!10, draw=blue!80] (2, 0) circle (2pt);
\filldraw[fill=blue!10, draw=blue!80] (1, 2) circle (2pt);
\filldraw[fill=blue!10, draw=blue!80] (0, 1) circle (2pt);
\filldraw[fill=blue!10, draw=blue!80] (3, 1.5) circle (2pt);
\filldraw[fill=blue!10, draw=blue!80] (2.5, 2) circle (2pt);
\filldraw[fill=blue!10, draw=blue!80] (1.2, 0.4) circle (2pt);
\filldraw[fill=red!10, draw=red!80,name intersections={of=1 and 2}] (intersection-1) circle (2pt);
\filldraw[fill=red!10, draw=red!80,name intersections={of=1 and 2}] (intersection-2) circle (2pt);
\filldraw[fill=red!10, draw=red!80,name intersections={of=1 and 3}] (intersection-1) circle (2pt);
\filldraw[fill=red!10, draw=red!80,name intersections={of=2 and 3}] (intersection-1) circle (2pt);
\end{tikzpicture}
\end{subfigure}
\caption{[Left] An abstract simplicial complex $\K$ with planar vertices has been depicted. 
[Middle] The shadow $\sh(\K)\subset\R^2$ has been shown as a subset of the plane. 
[Right] A triangulation of the shadow is shown. 
The new shadow vertices are shown in \textcolor{red}{red}.}
\label{fig:shadow}
\end{figure}
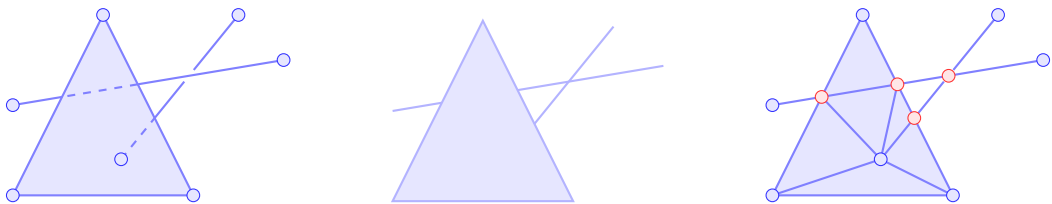

\subsection{Shadow of Complexes and Our Motivation}
%With the advent of modern sampling technologies, such as GPS, sensors, and medical imaging, Euclidean point clouds are increasingly available for analysis. 
Our theoretical study of Vietoris–Rips complexes and their shadows is motivated by the practical challenge of reconstructing the topology and geometry of a compact Euclidean ``shape'' from a finite, nearby point cloud ``sample''.
In practice, such point clouds typically lie on or near a simpler underlying shape $X\subset\R^N$; the sample $S\subset\R^N$ is described as \emph{noiseless} if it lies directly on the shape and \emph{noisy} otherwise. 
The relatively new field of \emph{shape learning} focuses on inferring the topological and geometric properties of the unknown shape $X$ from a finite point cloud $S$ sampled within the Hausdorff proximity (Definition~\ref{def:dH}) to $X$. 
\begin{definition}[The Hausdorff Distance]\label{def:dH}
Let $(X,d)$ be a metric space. 
Let $\mathcal A$ and $\mathcal B$ be compact, non-empty subsets. 
The \emph{Hausdorff distance} between them, denoted $d_\h(A, B)$, is defined as
\[
d^X_\h(A, B) \coloneqq 
\max\left\{\sup_{a\in A}\inf_{b\in B}d(a,b),\sup_{b\in B}\inf_{a\in A}d(a,b)\right\}.
\]
In case $X\subset\mathbb{R}^N$ and $ A, B,\mathcal{X}$ are all equipped with the Euclidean metric, we simply write $d_\h( A, B)$.
\end{definition}
In the last decade, the problem of shape reconstruction has received far and wide attention both in theoretical and applied literature; see, for instance
\cite{Amenta-Bern-Eppstein:1998,Dey:2006:CSR:1196751,Niyogi-Smale-Weinberger:2008,Chazal-Lieutier:2007,Chazal-Oudot:2008,Chazal-Cohen-Steiner-Lieutier:2009,kim2020homotopy,fasy2022reconstruction,majhi2024demystifying}.
In order to reconstruct an unknown shape $X$, the sample $S$ is commonly ``interpolated'' to compute a replica or ``reconstruction'' $\widehat{X}$ that is equivalent to $X$ in some appropriate sense (e.g., homotopy equivalent, homeomorphic, etc). 
The developments in shape reconstruction can be classified into two broad objectives: \emph{topological} and \emph{geometric}. 
While topological reconstruction concerns estimating only the topological features (e.g., homology/homotopy groups) of the underlying shape $X$ by computing an abstract topological object $\widehat{X}$ that is only topologically faithful (homotopy-equivalent) to $X$. 
To produce $\widehat{X}$, the Vietoris--Rips complex $\Ri_\beta(S)$---on the sample $S$ at an appropriate scale $\beta$---is commonly used in the topological data analysis community. 
Examples of homotopy-equivalent reconstruction results using Vietoris--Rips complexes include~\cite{majhi2023vietoris,majhi2024demystifying,MajhiStability,kim2020homotopy,attali2011vietoris}.

Topologically faithful reconstructions are primarily used to estimate the homological features of a hidden shape $X$, such as its Betti numbers and Euler characteristic.
However, the more ambitious paradigm of geometric reconstruction seeks to compute a subset of $\mathbb{R}^{N}$---a \emph{geometric embedding}---that is both (a) topologically faithful (homeomorphic or homotopy equivalent) and (b) geometrically close (in the Hausdorff distance) to $X$.
While abstract Vietoris–Rips complexes facilitate homotopy-equivalent reconstructions, they do not inherently provide an embedding within the host Euclidean space. 
For the geometric reconstruction of Euclidean shapes, it is more natural to consider the shadow of these complexes (as defined below).
In their recent work~\cite{majhi_shadow} on Euclidean graph reconstruction, the authors provide 
a provable algorithm leveraging the shadow of Vietoris--Rips complexes of a Hausdorff-close sample as the geometric embedding of the underlying graph.

\begin{definition}[Shadow]\label{def:shadow}
Let $\K$ be an abstract simplicial complex with vertices in $\R^N$, i.e., $\K^{(0)}\subset\R^N$. 
The \emph{shadow projection} map $p\colon\K\to\mathbb{R}^N$ sends a vertex $v\in\K^{(0)}$ to the corresponding point in $\mathbb{R}^N$, then extends linearly to all points of the geometric realization (abusing notation still denoted by) $\K$.
We define the \emph{shadow} of $\K$ as its image under the projection map $p$, i.e.,
\[
\sh(\K)\coloneqq
\bigcup_{\sigma=[v_0,v_1,\ldots,v_k]\in\K}\mathrm{Conv}(\sigma),
\]
where $\mathrm{Conv}(\cdot)$ denotes  the convex hull of a subset in $\mathbb{R}^N$. 
\end{definition}
Since the shadow is a polyhedral subset of $\R^N$, it can be realized by an at most $N$-dimensional simplicial complex, yet it may not admit a canonical triangulation.  Figure~\ref{fig:shadow} illustrates one such triangulation.

In the particular context of Vietoris--Rips complex of a Euclidean sample $S$, a study of geometric/topological behavior of the canonical projection map $p\colon\Ri_\beta(S) \to \sh(\Ri_\beta(S))$ is a natural yet deceptively non-trivial problem. 
The map $p$ often possesses complex ``singularities'' that have been resolved only for low-dimensional $\mathbb{R}^N$, $N\leq 3$~\cite{chambers_vietorisrips_2010,adamaszek_homotopy_2017}.  
The obstacle naturally leads us to shift our focus to the {\it system} of complexes: 
\[
\{ \sh(\Ri_{\beta}(S)) \mid \beta > 0, S\text{ finite subset of } X\}.
\]
This paper investigates the homotopy properties of the system. 
It turns out that in the limit ``$\beta \to 0$ and $S \to X$'', the ``limit map $\lim p$'' is well behaved with respect to homotopy/homology groups.

Shape theory (\cite{borsuk1975theoryofshape}, \cite{dydaksegal1978shapetheory} and \cite{MardesicSegal}), a homotopy theory for non-ANR 
spaces, provides us with a convenient framework. Although the spaces we study are mostly ANR spaces, the inverse system approach developed in \cite{MardesicSegal} provides a convenient language for their study. It is indeed possible to formulate our results in terms of the category Pro-HTOP (\cite{MardesicSegal}). Such a formalism is finer than that of our statement, yet we stay in the present form to avoid the technicality.

\subsection{Problem Setup}\label{problemSetup}
In the context of Euclidean shape reconstruction, the unknown underlying space \(M\subset \mathbb{R}^{N}\) is conveniently modeled as a Riemannian submanifold or an embedded graph, and the sample as a finite subset $S\subset\R^N$.

To facilitate a more general reconstruction framework, we consider $M$ to be a compact connected metric space in $\mathbb R^{N}$ with the induced metric $d_M$ from the Euclidean length structure.  
%{\color{blue}[[ I feel that we should give the definition of length spaces here.  Also: the "length space" assumption is a little stronger than necessary.  It suffices to assume the geodesic-property locally.  I wonder if you use the word "a geodesically complete length space" below in that sense??  Geodesically completeness seems to be used for Riemannian manifolds.  In "stability paper" of Sush, "complete length spaces" are used. ]]}
We assume that the metric space $(M,d_{M})$ is a %complete
length space, that is, a %complete 
metric space such that, for each pair $p,q$ of points of $M$, there exists an isometry, called a geodesic, $c\colon [0,d_{M}(p,q)] \to M$ such that $c(0) = p,c(d_{M}(p,q)) = q.$  Furthermore, we assume that $M$ satisfies the conditions (M1)--(M3) as stated below. Throughout,  $\|\bullet\|$ denotes the standard Euclidean norm on $\mathbb R^N$.
%Riemannian submanifolds, metric graphs with $\varepsilon$-path length and CAT($\kappa$) subspace of $\mathbb R^N$ satisfy these conditions. 

\subsubsection*{Assumptions}\label{assumptions}

Let $(M,d_{M})$ be a compact metric space of $\mathbb R^{N}$ that admits a neighborhood $N(M)$ of $M$ and a retraction $\pi\colon N(M) \to M$. For $r>0$, let $N_{r}(M) = \{ x\in \mathbb{R}^N \mid \inf_{p\in M} \| x-p\| \leq r\}$.
We assume that
\begin{itemize}
\item[(M1)] there exists a $\rho(M) > 0$ such that any two maps $f,g\colon X\to M$ of a space $X$ to $M$ satisfying $d_{M}(f(x),g(x))< \rho(M)$ for each $x\in X$ are homotopic: $f \simeq g$.
\item[(M2)] there exist $\delta>0$ and $\xi \in (1,\infty)$ such that
for each $p,q\in M$ with $\Vert p-q\Vert < \delta$, we have
\[
\Vert p-q \Vert \leq d_{M}(p,q) \leq \xi \Vert p-q\Vert.
\]
\item[(M3)] for each $r>0$ with $N_{r}(M) \subset N(M)$, there exist an $\varepsilon_{r}>0$ with $\displaystyle \lim_{r\to 0}\varepsilon_{r} =0$ such that
\[
\Vert \pi(x) - x \Vert < \varepsilon_r
\]
for each $x\in N_{r}(M)$.
\end{itemize}
As shown in \cite[Section 4]{majhi2024demystifying} and \cite[Section 4]{majhi2023vietoris}, closed Euclidean submanifolds with induced metrics and compact Euclidean embedded graphs with finitely many edges with $\varepsilon$-path metrics (with small $\varepsilon$) satisfy the above conditions.
%In addition, $(M, d_M)$ a complete length space.

We denote by $\Ri^{\R^N}_\beta(S)$, $\Ri_\beta(M)$, and $\Ri^{\R^N}_\beta(M)$ the Vietoris--Rips complexes of $(S, \|\bullet\|)$, $(M,d_M)$, and $(M, \|\bullet\|)$, respectively, to ask most natural questions:
\begin{itemize}
\item[(a)] [\textbf{Hausmann-Type}] Is it true that 
$\sh(\Ri_\beta (M))$ or $\sh(\Ri^{\R^N}_\beta (M))$ is homotopy equivalent to $M$
for any sufficiently small $\beta>0$?

\item[(b)] [\textbf{Latschev-Type}] Is it true that 
$\sh(\Ri^{\R^N}_\beta (S))$ is homotopy equivalent to $M$ for any sufficiently small $\beta>0$ and for any sample set $S$ that is sufficiently Hausdorff distance-close to $M$?

\item[(c)] [\textbf{Shadow Projection}] Is it true that the map $p\colon\Ri^{\R^N}_\beta(S) \to \sh(\Ri^{\R^N}_\beta (S))$ is a homotopy equivalence for any sufficiently small $\beta>0$ and for any sample set $S$ that is sufficiently Hausdorff distance-close to $M$?
\end{itemize}
We do not know the answer to the above question in its full generality, but we show that the above are valid when we take appropriate direct and inverse limits with respect to $S$ and $\beta$.

\subsection{Outline}
Section~\ref{sec:preli} presents the relevant definitions of properties of direct and inverse systems of groups.
In Section~\ref{sec:VRLimit}, we prove very natural limit properties, such as Theorem~\ref{thm:weaktop}, for Vietoris--Rips complexes of a general (abstract) metric space.
Section~\ref{sec:shadowLimit} and Section~\ref{sec:shadowLimitNoisy} consider limits of the shadow of the Vietoris--Rips complexes of $M\subset\R^N$ (see~\ref{assumptions}) for the noiseless and noisy samples, respectively.
The main limit theorems of Section~\ref{sec:shadowLimit} are Theorem~\ref{thm:limitpi} and Theorem~\ref{limitproj} for the Vietoris--Rips shadow and the canonical projection map, respectively.
Section~\ref{sec:shadowLimitNoisy} correspondingly present in the noisy analogs, respectively, in Theorem~\ref{limitpi2} and Theorem~\ref{limitprojection2}.
Finally, Section~\ref{sec:recon} demonstrates in Theorem~\ref{onedim} the potential of Vietoris--Rips shadow for the geometric reconstruction of closed curves. We mostly focus on homotopy groups, but all results hold for homology groups as well.

\subsection{Notation}

Here we fix the notation used throughout the present paper.

$\Ri_\beta(X)$ denotes the Vietoris--Rips complex of a metric space $(X,d_{X})$ and let $\Ri_\beta^{\R^N}(S)$ be 
the Vietoris--Rips complex of $S\subset\R^N$ under the Euclidean metric. Geometric realizations endowed with the Whitehead topology are denoted by the same symbol for simplicity. 
For a simplicial complex $\K$ with $\K^{(0)} \subset \mathbb R^N$, $\sh(\K)$ denotes the shadow of the complex $\K$ and the shadow projection is denoted by
$p\colon \K \to \sh(\K)$.\\
%\hfill\\
%$\Ri_\beta(X)$= the Vietoris--Rips complex of a metric space $(X,d_{X})$, and \\
%$\Ri_\beta^{\R^N}(S)$ = the Vietoris--Rips complex of $S\subset\R^N$ under the Euclidean metric. We assume that both complexes are endowed with the Whitehead topology. \\
%For a simplicial complex $\K$ with $\K^{(0)} \subset \mathbb R^N$, $\sh(\K)$= the shadow of the complex $\K$.\\ 
%$p\colon \K \to \sh(\K)$= the shadow projection .\\

%For $M\subset\mathbb{R}^N$ satisfying~\ref{assumptions}:\\
%$\Vert \bullet \Vert$ = the Euclidean distance of $\mathbb{R}^N$.\\
%$d_{M}$= the induced length metric on $M$\\
%$N(M)$= a tubular neighborhood of $M$ and\\
%$N_{\tau}(M)$ = a tubular $\tau$-neighborhood of $M$\\
%$\pi\colon N(M)\to M$ = the bundle projection.  We assume that $N(M)$ is sufficiently close to $M$ so that, for each $x\in N(M)$, we have
%\[
%\Vert \pi(x)-x\Vert = \min_{p\in M}\Vert p-x\Vert.
%\]
%$\tau$ is assumed to be small as well so that $N_{\tau}(M) \subset N(M)$.\\

For a subset $A$ of $M$, $\diam_{M}(A)$ denotes the diameter of $A$ with respect to the metric $d_M$.  For a subset $B$ of $\mathbb{R}^N$, $\diam_{\mathbb{R}^N}(B)$ denotes the Euclidean diameter of $B$.  
For $B\subset M$, $\conv(B)$ and $\conv_M(B)$ denote, respectively, the Euclidean and geodesic convex hulls of $B$.
For a continuous map $f\colon X\to Y$ between spaces $X$ and $Y$, the induced homomorphism between the homotopy groups is also denoted by $f:\pi_{m}(X) \to \pi_{m}(Y)$ for simplicity.

\section{Preliminaries on Direct and Inverse Systems of Groups}\label{sec:preli}
This section presents essential notation, definitions, and properties of direct and inverse systems of groups, as well as their limits. 

\begin{definition}[Direct Systems of Groups]\label{def:direct-l}
Let $\{G_\alpha\}_{\alpha \in \mathcal{I}}$ be a family of groups indexed by a partially ordered set $\mathcal{I}$ and, whenever $\alpha \preceq \beta $,  $f_{\alpha,\beta}\colon  G_\alpha \to G_\beta$ be a homomorphism such that: 
\begin{enumerate}[(i)]
\item $f_{\alpha, \alpha}$ is the identity homomorphism,
\item whenever $\alpha \preceq \beta \preceq \gamma$,  $f_{\alpha, \gamma} = f_{\beta, \gamma} \circ f_{\alpha, \beta}$, and
\item for every $\alpha, \beta \in \mathcal{I}$ there is $\gamma \in \mathcal{I}$ such that $\alpha, \beta \preceq \gamma$.
\end{enumerate}
Then $\{G_\alpha, f_{\alpha,\beta}\}$ is called a \emph{direct system} of groups and homomorphisms. 

The \emph{direct limit}, denoted $\varinjlim G_\alpha$, is the set of equivalence classes $[\cdot] $ on the disjoint union $\sqcup G_\alpha / \sim$, where the equivalence relation $\sim$ is generated by 
\[
x_{\alpha} \sim y_{\beta} ~(x_\alpha \in G_{\alpha}, y_\beta \in G_\beta) \Leftrightarrow f_{\alpha,\gamma}(x_{\alpha}) = f_{\beta,\gamma}(y_{\beta}) ~\mbox{for some}~\gamma \succeq \alpha,\beta.
\]
The group operation is defined by
\[
[x_{\alpha}] \cdot [y_{\beta}] = [f_{\alpha,\gamma}(x_{\alpha})\cdot f_{\beta,\gamma}(y_{\beta}) ]
\]
which is well defined due to (3) above.
Each $G_{\alpha}$ admits the canonical homomorphism $f_{\alpha}\colon G_{\alpha} \to \varinjlim G_\alpha$.  
The system $\{f_{\alpha}\colon G_{\alpha} \to \varinjlim G_\alpha\}$ forms the colimit in the category of groups.
\end{definition}

As a relevant example, one can consider $\mathbb{S}$ to be the partially ordered set of all non-empty finite subsets $S$ of a metric space $(X,d_X)$, ordered by set inclusion.
For a fixed scale $\beta>0$, note the natural inclusion between the Vietoris--Rips complexes $\iota_\beta^{S,T}\colon\Ri_\beta(S)\to\Ri_\beta(T)$ for any $S\subseteq T$.
So, for any $m\geq0$, the family of homotopy groups of the Vietoris-Rips complex $\left\{\pi_m(\Ri_\beta(S)), \iota_\beta^{S,T}\right\}$ will form a direct system.
The direct limit of the system is discussed in Remarks~\ref{rem:VR1} and \ref{rem:VR2}.

The direct limit group has the following characterization.
Let $\{G_\alpha, f_{\alpha,\beta}\}$ be a direct system of groups and assume that a homomorphism system $\{h_{\alpha}\colon G_\alpha \to H\}$ is given to a group $H$ so that $h_{\alpha} = h_{\beta}\circ f_{\alpha,\beta}$ for each $\alpha \preceq \beta$. Then the induced homomorphism
\[
\varinjlim_\alpha h_\alpha\colon\varinjlim G_\alpha \to H
\]
is an isomorphism if and only if
\begin{enumerate}
    \item for each $x\in H$, there exist $\alpha$ and $x_\alpha \in G_\alpha$ such that $h_{\alpha}(x_{\alpha}) = x$, and
    \item if $x_{\alpha}\in G_\alpha$ satisfies $h_{\alpha}(x_{\alpha}) = 1$, then there exists $\beta$ with $\alpha \preceq \beta$ such that $f_{\alpha,\beta}(x_{\alpha}) = 1$.
\end{enumerate}

\begin{definition}[Inverse Systems of Groups]\label{def:inverse-l}
Let $\{G_\alpha\}_{\alpha \in \mathcal{I}}$ be a family of groups indexed by a partially ordered set $\mathcal{I}$ and, whenever $\alpha \preceq \beta$, $f_{\alpha,\beta}\colon G_\beta \to G_\alpha$ be a homomorphism such that
\begin{enumerate}
    \item $f_{\alpha, \alpha}$ is the identity homomorphism,
    \item whenever $\alpha \preceq \beta \preceq \gamma$, $f_{\alpha,\gamma} = f_{\alpha, \beta} \circ f_{\beta,\gamma}$, and
    \item for every $\alpha, \beta \in \mathcal{I}$ there is $\gamma \in \mathcal{I}$ such that $\alpha, \beta \preceq \gamma$.
\end{enumerate}
Then $\{G_\alpha, f_{\alpha,\beta}\}$ is called an \emph{inverse system} of groups and homomorphisms, and the \emph{inverse limit} $\varprojlim G_\alpha$ is
the subset
\[
\left\{ (x_{\alpha}) \in \prod G_\alpha \mid
f_{\alpha,\beta}(x_\beta) = x_\alpha,~\mbox{whenever}~ \alpha \preceq \beta\right\}. 
\]
with the component-wise group operation. 
For each $\alpha$, the projection $f_\alpha\colon\varprojlim_{\alpha}G_{\alpha} \to G_\alpha$ is defined.  The system $\{ f_{\alpha}\colon\varprojlim_{\alpha}G_{\alpha} \to G_\alpha \}$ forms the limit in the category of groups.

%of the product $\prod G_\alpha$ consisting of all tuples $(g_\alpha)$ such that $f_{\beta,\alpha}(g_\beta) = g_\alpha$ whenever $\alpha \preceq \beta$. The group operation on the inverse limit is component-wise.
\end{definition}

%We similarly state a characterization of the inverse limit group.
Let $\{G_\alpha, f_{\alpha,\beta}\}$ be an inverse system of groups and assume that a homomorphism system $\{h_{\alpha}\colon H\to G_\alpha\}$ is given from a group $H$ so that $h_{\alpha} = f_{\alpha,\beta}\circ h_{\beta}$ for each $\alpha \preceq \beta$. 
Let $f_\alpha\colon\varprojlim G_\alpha\to G_\alpha$ be the projection.
Then we have the induced homomorphism:
\[
\varprojlim_\alpha h_\alpha\colon H\to \varprojlim G_\alpha
\]
which satisfies $h_\alpha=f_\alpha\circ \varprojlim_\alpha h_\alpha$ for each $\alpha$. Explicitly, 
$\varprojlim_\alpha h_\alpha\ $ is defined by
\[
\left(\varprojlim_\alpha h_\alpha\right)\ (x) = (h_{\alpha}(x))_{\alpha},~~x\in H.
\]
%{\color{blue} Here, can we add a characterization of inverse limits in the same spirit as direct limits above? k: I do not know such a characterization other than the universality (in the category sense).  Direct limits are defined via quotient, which makes the above characterization meaningful.  While inverse limits are described by subsets ("=" dual to quotient), and isomorphisms ("=" bijectivity) could be verified in point-set argument.???}

\section{Vietoris--Rips Limit theorems}\label{sec:VRLimit}
In this section, we present limit theorems for Vietoris--Rips complexes.
Throughout this section, $(X, d_X)$ represents an arbitrary metric space. Our focus is on Vietoris-Rips complexes, but
these results extend to related simplicial constructions, such as \v{C}ech and $\alpha$-complexes \cite{edelsbrunner2022computational}.

The $m$-dimensional sphere and $(m+1)$-dimensional ball are denoted by $S^m$ and $D^{m+1}$ respectively.

\begin{proposition}\label{prop:VRlimit}
Let $X$ be a non-empty set and $\{\K(S)\mid S=\K(S)^{(0)}, S\subset X, \text{ finite}\}$ be a family of simplicial complexes such that there is an inclusion $\iota^{S,T}\colon\K(S)\to\K(T)$ if $S\subset T$. Then the natural inclusion
$\iota^{S}\colon\K(S) \to \K(X)$ induces an isomorphism
\[
\varinjlim_{S}\iota^{S}\colon \varinjlim_{S} \pi_m(\K(S)) \to \pi_m(\K(X))
\]
for each $m\geq 0$.
\end{proposition}

\begin{remark}\label{rem:VR1}
For a metric space $(X,d_X)$ and scale $\beta>0$, as a corollary, one can take $\K(S)=\Ri_\beta(S)$ to prove that $\varinjlim_{S} \pi_m(\Ri_\beta(S)) \cong \pi_m(\Ri_\beta(X))$
for each $m\geq 0$.
\end{remark}

\begin{proof}
Our proof is a straightforward modification of a standard fact on CW complexes.
%The direct limit $\varinjlim \pi_m(\K(S))$ is isomorphic to $\pi_m(\K(X))$, because 
The complex $(\K(X)$ has the Whitehead topology, that is the weak topology with respect to the simplices (with the standard topology):
a subset $F$ is closed if and only if $F\cap \sigma$ is closed in $\sigma$ for each simplex $\sigma$.
We first prove the following claim.\\
\textbf{Claim}: Every compact set $F$ of $\K(X)$ is contained in a finite subcomplex of $\K(X)$.\\

\noindent\textbf{Proof of Claim}: We first show that the set:
\[
\mathcal{F}\coloneqq\{ \sigma\mid \operatorname{Int}\sigma \cap F \neq \emptyset\}
\]
is finite. Here, $\operatorname{Int}\sigma$ denotes the simplex-interior (not the topological interior in the whole space) of $\sigma$.
We prove by contradiction. 

We suppose the contrary. Then, there are infinitely many simplices $\sigma_i$ and points $x_{i}\in \operatorname{Int}\sigma_{i} \cap F$.  Let $I\coloneqq \{x_{i}\}$. 
For each $\sigma$ of $\K(X)$, we have $\operatorname{Int}\sigma_{i}\cap \sigma \neq \emptyset \Rightarrow \sigma_{i} \subset \sigma$, i.e., $\sigma_i$ is a face of $\sigma$.  Hence, there are only finitely many $i$'s such that $\sigma_{i} \subset \sigma$. This implies that $I \cap \sigma$ is a finite set and, in particular, is a closed subset of $\sigma$.  This means that $I$ is a closed subset (by the definition of the topology) of $F$ and hence is compact.  The same proof shows that every subset of $I$ is closed. In other words, $I$ is a discrete space and therefore cannot be an infinite set by the compactness of $F$, which is a contradiction. This proves the claim.

The above claim implies the following inclusion:
\[
F\subset \bigcup_{\sigma\in \mathcal F}\operatorname{Int}\sigma 
\subset \bigcup_{\sigma\in \mathcal F}\sigma.
\]
Let $S$ be the set of all vertices of $\sigma \in \mathcal F$. Then $F \subset \K(S)$.

Using the above, we can show the geometric versions of the characteristic properties of direct limit as stated right after Definition~\ref{def:direct-l}.
\begin{itemize}
\item[(i)] For each map $f\colon S^{m}\to \K(X)$, there exists a finite subset $S$ of $X$ such that $\operatorname{Im}(f) \subset \K(S)$.
\item[(ii)] If a map $g\colon S^{m}\to \K(S)$, where $S$ is a finite subset of $X$, admits an extension $\bar{g}\colon D^{m+1} \to \K(X)$, then there exists a finite subset $S' \supset S$ such that $\operatorname{Im}(\bar{g}) \subset \K(S')$.
\end{itemize}
These two are characteristic properties of the direct limits, and the conclusion follows.
\end{proof}

The next theorem slightly generalizes the above in the following sense:
%is slightly generalized in the theorem below in the following sense: 
rather than considering all finite subsets, the same direct limit is obtained by successively adding points.

%Using the same idea as before, we prove the following result.

\begin{theorem}\label{thm:weaktop}
Assume that, for each separable space $Z$, there associates a simplicial complex $\K(Z)$ which satisfies the following condition:
\begin{enumerate}
\item$Z = \K(Z)^{(0)}$.
\item If $Z_{1} \subset Z_{2} \subset Z$, we have the inclusion
$\iota^{Z_{1},Z_{2}}\colon\K(Z_{1}) \to \K(Z_{2})$ and moreover,  
$\K(Z_{1}) = \{ \sigma\in \K(Z_{2})\mid \sigma^{(0)} \subset Z_{1} \}$. 
%{\color{red}maybe $Z_1, Z_2$ as we have $S_k$ later?}
\item For each simplex $\sigma = [z_{0},\ldots,z_{n}]$ of $\K(Z)$ with vertices $z_{0},\ldots, z_{n}$, there exists an open neighborhood $U$ of $\{z_{0},\ldots, z_{n}\}$ in $Z$ such that for each finite set $\tau=\{w_{1},\ldots,w_{m}\} \subset U$,
the points of $\sigma\cup\tau$ span a simplex of $\K(Z)$.
\end{enumerate}
Let $D\coloneqq\{ p_{k}\mid k=1,2,\ldots\}$ be a countable dense subset of a separable space $X$ and let $S_{k}\coloneq \{ p_{i}\mid i=1,\ldots, k\}$. 
%$\{\K(S_{k}) \mid k\in\mathbb{N}\}$ be a family of simplicial complexes such that there is an inclusion $\iota^{k,l}\colon\K_k\to\K_l$ if $k\leq l$.  
Then the system of inclusions $\{ \iota^{S_{k},X}\colon \K(S_{k}) \to \K(X) \mid k=1,2,\ldots \}$
induces an isomorphism:
\begin{equation}\label{eq:samplelimit}
\varinjlim_{k} \iota^{S_{k},X}\colon \varinjlim_{k}\pi_m(\K(S_{k})) \to \pi_m(\K(X)).
\end{equation}
\end{theorem}

\begin{remark}\label{rem:VR2}
For a metric space $(X,d_X)$ and scale $\beta>0$, as a corollary, one can take $\K(S_k)=\Ri_\beta(S_k)$ to prove that $\varinjlim_{k} \pi_m(\Ri_\beta(S_k)) \cong \pi_m(\Ri_\beta(X))$
for each $m\geq 0$.
From our definition of Vietoris-Rips complexes, the diameter of each simplex in $\Ri_\beta(S_k)$ is \emph{strictly} less than $\beta$. Thus, condition (3) above is indeed satisfied.
For more on the distinction between `$<$' and `$\leq$' in the definition of Vietoris--Rips complexes, see~\cite{Adams_2019}.
\end{remark}

\begin{remark}\label{rem:VRThicken}
The metric thickening, denoted by $\Ri_{\beta}^{\mathfrak{M}}(X)$ in the present paper, of $(X,d_{X})$ with scale parameter $\beta$ was introduced in~\cite{Adamaszek2018}. 
There exists a natural continuous bijection $j\colon\Ri_{\beta}(X) \to \Ri^{\mathfrak{M}}_{\beta}(X)$ that induces an isomorphism in homotopy groups in all dimensions~\cite[Theorem 1]{gillespie2024}.  
For each finite subset $S$ of $X$, $\Ri^{\mathfrak{M}}_\beta(S)$ is homeomorphic to $\Ri_\beta(S)$, due to the compactness of $\Ri_\beta(S)$.
Combining these two, we see that the $\varinjlim \pi_m(\Ri^\mathfrak{M}_\beta(S))\cong\pi_m(\Ri^\mathfrak{M}_\beta(X))$ holds also for the metric thickening. 

\end{remark}

We start with a lemma.
\begin{lemma}\label{lem:weaktop}
Let $D$, $X$, $\K(\bullet)$, and $S_k$ be as defined in Theorem~\ref{thm:weaktop}.
Let $S$ be a finite subset of $X$.
Let $(P,Q)$ be a pair\footnote{Here, $Q$ is a subcomplex of $P$ with respect to a triangulation of $P$} of compact polyhedra and let $F\colon P \to \K(S)$ be a continuous map satisfying the following condition:
\begin{enumerate}
\item there exists a triangulation $T_Q$ of $Q$, an integer $k$ and $S_{k} \subset S$ such that $f\coloneqq F\lvert_Q\colon T_{Q} \to \K(S_{k})$ is a simplicial map.
\end{enumerate}
Then, there exist an integer $\ell > k$, a triangulation
%\footnote{$T_P$ extends $T_Q$ k: the meaning "extends" is explained in the phrase "which contains $T_Q$ as a subcomplex"} 
$T_P$ of $P$ which contains $T_Q$ as a subcomplex and a simplicial map $G\colon T_{P} \to \K(S_{\ell})%\cup S)
$ such that
$G \simeq F$ rel. $Q\colon P\to  \K(S\cup S_{\ell})$.
\end{lemma}

\begin{proof}
Since $D$ is dense, using the assumption (3) of Theorem~\ref{thm:weaktop} and the finiteness of $S$ we observe the following: for each point $x\in S$, we can choose a point $p_{x}\in D$ such that, for any $x_{0},\ldots,x_{n} \in S$
\begin{equation}\label{closeness}
[x_{0},\ldots,x_{n}] \in \K(S) \Rightarrow 
[x_{0},\ldots,x_{n},p_{x_0},\ldots,p_{x_n}] \in \K(S\cup\{p_{x_0},\ldots,p_{x_n}\}).
\end{equation}
We take a sufficiently large $\ell>k$ such that 
\[
\{p_{x} \mid x\in S\} \subset S_\ell.
\]
By the Relative Simplicial Approximation Theorem (cf. \cite[the paragraph after Theorem 2C.1]{hatcher2002book}, we may find a triangulation $T_P$ of $P$ that contains $T_Q$ as a subcomplex and a simplicial map $\Phi\colon T_{P}\to \K(S)$ such that
\[
\Phi \simeq F  ~~rel. T_{Q}.
\]
In particular $\Phi\lvert_{T_{Q}} = f$.  For each vertex $v\in T_P$, we define
$G(v)$ by
\begin{equation}\label{G}
G(v) = p_{\Phi(v)}, 
\end{equation}
where
\[
~~\mbox{if}~v\in T_{Q},~\mbox{then we choose}~G(v) = f(v) \in S_{k} \subset S_{\ell}.
\]
If $\sigma=[v_{0},\ldots, v_{n}]$ is a simplex of $T_P$, then $[\Phi(v_{0}),\ldots, \Phi(v_{n})] \in \K(S)$.
%, because $\Phi(\sigma)$ is a simplex of $\Ri_\beta(S)$.  
From (\ref{closeness}), we obtain 
\[
[\Phi(v_{0}),\ldots, \Phi(v_{n}), p_{\Phi(v_{0})},\ldots,p_{\Phi(v_{n})}] \in \K(S\cup S_\ell).
\]
In particular, $G$ on the vertices $T_P^{(0)}$ defined by (\ref{G}) induces a simplicial map $G\colon T_{P}\to \K(S_{\ell})$. In addition, the above shows that the set of vertices $\{\Phi(v_{0}),\ldots, \Phi(v_{n}), p_{\Phi(v_{0})},\ldots,p_{\Phi(v_{n})}\}$ spans a simplex of $\K(S \cup S_{\ell})$. Hence $\Phi$ and $G$ are contiguous simplicial maps to  $\K(S \cup S_{\ell})$ and $\Phi\lvert_{T_{Q}} = f\lvert_{T_{Q}}$.  Hence 
\[
G\simeq \Phi \simeq F~rel.~Q.
\]
This proves the lemma.
\end{proof}

We now provide the proof of Theorem~\ref{thm:weaktop}.
\begin{proof}[Proof of Theorem~\ref{thm:weaktop}]
In order to prove~\eqref{eq:samplelimit}, we take an arbitrary map $F\colon S^{m}\to \K(X)$. Following the proof of Proposition~\ref{prop:VRlimit}, there exists a finite subset $S$ of $X$ such that $\operatorname{Im}(F) \subset \K(S)$.  Applying Lemma~\ref{lem:weaktop} to the pair  of polytopes $(P,Q) = (S^{m},\emptyset)$, we find an integer $\ell$ and $G\colon S^{m}\to \K(S_{\ell})$ such that $G\simeq F\colon S^{m} \to \K(S\cup S_{\ell}) \to \K(X)$.

%\sush{Do we need the above paragraph? k: yes we do.  It is a surjectivity proof.}

Now let us assume that a map $f\colon S^{m}\to \K(S_{k})$ is given so that $f\simeq 0\colon S^{m}\to \K(X)$, i.e., $f$ is null homotopic in $\K(X)$. 
Taking a simplicial approximation, we can assume at the beginning that $S^{m}$ has a triangulation, denoted by $T_{S^m}$, and $f\colon T_{S^m} \to \K(S_{k})$ is a simplicial map.  

Since the map $f$ admits an extension $F\colon D^{m+1}\to \K(X)$, following the proof of Proposition~\ref{prop:VRlimit}, we can find a finite subset $S$ of $X$ such that $S\supset S_k$ and $\operatorname{Im}(F) \subset \K(S)$.  

Applying Lemma~\ref{lem:weaktop} to the pair $(P,Q) = (D^{m+1},S^{m})$ and $(F,f)$, we find an integer $\ell$ and a map $G\colon D^{m+1}\to \K(S_{\ell})$ that is an extension of $f$.  
Hence, $f$ is null homotopic as a map $S^{m}\to \K(S_{\ell})$ in $\K(S\cup S_\ell)$, therefore in $\K(X)$.
This proves~\eqref{eq:samplelimit}.
\end{proof}

% We may find later a more suitable spot for the following remark
%\begin{remark}
%In case $\X$ is a closed Riemannian manifold.
%What's the relation between $\varinjlim \pi_m(\Ri^m_\beta(S))$ and $\pi_m(\Ri^m_\beta(\X))$ when considered the metric thickening? 
%By a theorem due to Gillespie \cite{gillespie2024}, the natural map 
%$\Ri_{\beta}(\X) \to \Ri^{m}_\beta(\X)$ is a weak homotopy equivalence, which answers the above question.
%\end{remark}

\section{Shadow Limit Theorems for noiseless samples}\label{sec:shadowLimit}

Let us denote by $\mathbb{S}$ the directed set of all finite subsets of $M$ ordered by inclusion.
For $S, T \in \mathbb{S}$ and $\beta>0$, 
let
\[
\iota_{\beta}^{S,T}\colon \sh(\Ri_\beta (S)) \to \sh(\Ri_\beta (T)) 
\]
be the inclusion.  
To simplify the notation, the induced homomorphism in homotopy groups is also denoted by $\iota_{\beta}^{S,T}\colon \pi_{m}(\sh(\Ri_\beta (S))) \to \pi_{m}(\sh(\Ri_\beta (T)))$.

Taking the $m$-homotopy groups of $\sh(\Ri_\beta (S))$ for  $S\in \mathbb{S},$ and homomorphisms induced by inclusions, we obtain a direct system
\[
\left\{\pi_{m}(\sh(\Ri_\beta (S))), \iota_{\beta}^{S,T}\colon \pi_{m}(\sh(\Ri_\beta (S))) \to \pi_{m}(\sh(\Ri_\beta (T))) \mid S,T \in \mathbb{S}, S\subset T \right\}
\]
for which the direct limit 
\[
\pi_{m}(\sh(\Ri_\beta (\mathbb{S}))) \coloneq  \varinjlim_{S\in \mathbb{S}} \pi_{m}(\sh(\Ri_\beta (S))) \]
with the canonical homomorphism $\iota_{\beta}^{S,\mathbb{S}}\colon \pi_{m}(\sh(\Ri_\beta (S))) \to \pi_{m}(\sh(\Ri_\beta (\mathbb{S}))$ is natually defined.
\begin{remark}
\begin{enumerate}
\item If the union
\[
\sh(\Ri_\beta (\mathbb{S})) \coloneq  \cup_{S\in \mathbb{S}} \sh(\Ri_\beta (S)) 
\]
is endowed with the weak topology with respect to the collection $\{ \sh(\Ri_\beta (S)) \mid S \in \mathbb{S} \}$, then $\pi_{m}(\sh(\Ri_\beta(\mathbb{S})))$ is isomorphic to the group $\pi_{m}(\sh(\Ri_\beta (\mathbb{S})))$, which justifies the above notation.
\item We may take homology groups to obtain a corresponding group for homology.  
%If we take cohomology, we take the projective limit in the above.
\end{enumerate}
\end{remark}

%{\color{blue} We also consider the following condition on the scale $\beta$:
%\begin{itemize}
%\item[($\beta$-1)] $\sh(\Ri_\beta (M)) \subset N(M)$, where $N(M)$ is the neighborhood of $M$ in (M1).
%\item[($\beta$-2)] $3 \beta < \eta(M)$,
%\item[($\beta$-3)] $\beta + \xi(\beta+\varepsilon_{\beta}) < \rho(M)$, where $\xi$ and $\varepsilon_{\beta}$ are constants in (M2) and (M3) for the induced metric $d_M$ on $M$ as a Riemannian submanifold of $\mathbb{R}^N$.
%\end{itemize}
%}

We fix $\beta_0$ such that $\sh(\Ri_{\beta_0} (M)) \subset N(M)$, where $N(M)$ is the neighborhood of $M$ as defined in~\ref{assumptions}~(M1), and assume throughout this section that $0< \beta < \beta_0$. 
When $\sh(\Ri_{\beta}(M)) \subset N(M)$, the restriction of $\pi\colon N(M)\to M$ to $\sh(\Ri_{\beta}(M))$ is denoted by $\pi_\beta$.  
%As in the previous sections, $\mathbb S$ and $\mathbb{S}_{\tau}$ denote the collections of all finite subsets of $M$ and $N_{\tau}(M)$ respectively.

For $S\in \mathbb{S}$, let 
\[
\pi_{\beta}^{S}\colon \sh(\Ri_\beta (S)) \to M
\]
be the restriction of $\pi\colon N(M)\to M$.  
For $S,T\in \mathbb S$ with $S\subset T$, 
we get the following commutative diagram:
% induced homomorphisms $\left\{ \pi_{\beta}^{S} \mid S\in \mathbb{S}\right\}$ makes the following diagram commutative:
\[
\xymatrix{
  \sh(\Ri_\beta (T))  \ar[rd]^{\pi_{\beta}^T}%& M 
\\
& M\\
\sh(\Ri_{\beta}(S)) \ar[uu]^{\iota_{\beta}^{S,T}}  \ar[ru]_{\pi_{\beta}^S}
%
%
%  \ar[dd]_{\phi^{-1}} & \mathcal{SH}(R_{\beta}(S)) \ar[d]^{r}\\
%& M' \ar[d]^{h}\\
%R_{\beta}^{L}(M) \ar[r]_{T} & M
%
%
%  G_{i+1} \ar[ld]_{\varphi_{i+1}} 
%                              \ar[dd]^{\bar{p}_{i}} \\
% G_{i} \ar[dd]_{p_{i}}  &                           \\
%   &  \Gamma_{f_{i+1}} \ar[ld]|{\varphi_{f_{i+1}}} \ar[d]^{p_{f_{i+1}}} \\
% X_{i}                    &  X_{i+1} \ar[l]_{f_{i+1}}           
}
\]
The above yields a corresponding commutative diagram for the $m$-homotopy groups, and we obtain the limit homomorphism
\[
\pi_{\beta}^{\mathbb S}\colon \pi_{m}(\sh(\Ri_{\beta}(\mathbb{S})))  \to \pi_{m}(M).
\]
For $0< \gamma < \beta$ and $S_{1},S_{2} \in \mathbb S$ with $S_1\subset S_2$, we have the commutative diagram:
\[
\xymatrix{
  \sh(\Ri_\gamma (S_{1}))  \ar[r]^{\iota_{\gamma}^{S_{1},S_{2}}} \ar[d]_{\iota_{\beta,\gamma}^{S_{1}}} & \sh(\Ri_\gamma(S_{2}) \ar[d]^{\iota_{\beta,\gamma}^{S_{s}}}
\\
\sh(\Ri_{\beta}(S_{1})) \ar[r]_{\iota_{\beta}^{S_{1},S_{2}}}  & \sh(\Ri_\beta(S_{2}))
}
\]
where all arrows indicate appropriate inclusions.
From the above, we see that, for $0< \gamma < \beta$, the direct limit of the inclusion $\iota_{\beta,\gamma}^{S}\colon \sh(\Ri_{\gamma}(S)) \to \sh(\Ri_{\beta}(S))$ 
%{\color{blue}is the notation $\iota_{\beta,\gamma}^S$ consistent with the convention of inverse systems? k: I followed the Mardesic-Segal notation because element-wise computation is mainly made for inverse limit, not for direct limits.  But yes, this may cause a confusion later.  Please let us decide which $i_{\beta,\gamma}$ or $i_{\gamma,\beta}$ to choose. Thank you!} 
induces the homomorphism
\[
\iota_{\beta,\gamma}^{\mathbb S}\colon \pi_{m}(\sh(\Ri_{\gamma}(\mathbb{S}))) 
\to \pi_{m}(\sh(\Ri_{\beta}(\mathbb{S}))),
\]
which makes the following diagram commutative:
\[
\xymatrix{
  \pi_{m}(\sh(\Ri_\gamma (\mathbb{S})))  \ar[rd]^{\pi_{\gamma}^{\mathbb{S}} }
\ar[dd]_{\iota_{\beta,\gamma}^{\mathbb S}}
\\
& \pi_{m}(M)\\
\pi_{m}(\sh(\Ri_{\beta}(\mathbb{S}))) \ar[ru]_{\pi_{\beta}^{\mathbb{S}}}
}
\]
We obtain an inverse system:
\[
\left\{ \pi_{m}(\sh(\Ri_{\beta}(\mathbb{S}))), \iota_{\beta,\gamma}^{\mathbb S}\colon \pi_{m}(\sh(\Ri_{\gamma}(\mathbb{S}))) \to \pi_{m}(\sh(\Ri_{\beta}(\mathbb{S}))) \mid 0 < \gamma < \beta <\beta_{0} \right\}
\]
and the inverse limit group 
\[
\varprojlim_{\beta} \pi_{m}(\sh(\Ri_{\beta}(\mathbb{S})))
\]
with the canonical homomorphism $\displaystyle \iota_{\beta,\infty}^{\mathbb S}\colon \varprojlim_{\beta} \pi_{m}(\sh(\Ri_{\beta}(\mathbb{S}))) \to \pi_{m}(\sh(\Ri_{\beta}(\mathbb{S})))$.

Moreover, we obtain a homomorphism:
\begin{equation}\label{eq:pi_inf_S_to_M}
\pi_{\infty}^{\mathbb S}\colon  \varprojlim_{\beta} \pi_{m}(\sh(\Ri_{\beta}(\mathbb{S}))) \to \pi_{m}(M),
\end{equation}
defined by
\[
\pi_{\infty}^{\mathbb S} = \pi_{\beta}^{\mathbb S} \circ \iota_{\beta,\infty}^{\mathbb S},
\]
where we observe that, if $\gamma < \beta < \beta_0$, then
\[
\pi_{\beta}^{\mathbb S}\circ \iota_{\beta,\infty}^{\mathbb S } 
=\pi_{\beta}^{\mathbb S}\circ \iota_{\beta,\gamma}^{\mathbb S } \circ \iota_{\gamma}^{\mathbb S}
= \pi_{\gamma}^{\mathbb S}\circ \iota_{\gamma,\infty}^{\mathbb S}.
\]
Thus, the above definition~\eqref{eq:pi_inf_S_to_M} does not depend on $\beta$.

Our first limit theorem for Vietoris--Rips shadow is stated as follows:
\begin{theorem}\label{thm:limitpi}
The homomorphism 
$\pi_{\infty}^{\mathbb S}\colon  \varprojlim \pi_{m}(\sh(\Ri_{\beta}(\mathbb{S}))) \to \pi_{m}(M)$ is an isomorphism for each $m\geq0$.
\end{theorem}

\begin{remark}
The same holds for homology.
\end{remark}

Before we give a proof of the theorem in Section~\ref{sec:Latschev}, we first present a special case thereof in the spirit of Hausmann's theorem~\cite[Theorem~3.5]{hausmann1995vietoris} for Vietoris--Rips complexes.

%\begin{remark}
%{\color{blue} (*) We do not need to assume that $M$ is a submanifold in the above %theorem.  It is enough to assume that
%\begin{itemize}
%\item[(1)] two sufficiently close maps to $M$ is homotopic,
%\item[(2)] the metric $d_M$ has bounded local distorsion with respect to the Euclidean distance.
%\item[(3)] there exists a neighborhood $N_M$ and a strong deformation retraction $\pi:N_{M}\to M$.
%\end{itemize}
%The assumption
%\begin{itemize}
%\item[(3')] there exists a neighborhood $N_M$ of $M$ such that, for each $x \in N_M$ there exists a unique point $\pi(x) \in M$ such that $\Vert x-\pi(x)\Vert = \min_{p\in M}\Vert x-p\Vert$, so that the map $\pi:N_{M} \to M$ is well defined,
%\end{itemize}
%will simplify the argument, but may be replaced with (3) above.  (I will check the detail once again).}
%\end{remark}

\subsection{Hausmann-Type Limit Theorem for Shadow}
The following proposition may be regarded as a corollary of Theorem~\ref{thm:limitpi} in essence, yet we give a proof prior to Theorem~\ref{thm:limitpi}, because it well demonstrates the idea of our argument of the present paper.

In light of the discussion above, we obtain an inverse system 
\[
\left\{ \pi_{m}(\sh(\Ri_{\beta}(M))), \iota_{\beta,\gamma}\colon \pi_{m}(\sh(\Ri_{\gamma}(M))) \to \pi_{m}(\sh(\Ri_{\beta}(M))) \mid 0 < \gamma < \beta <\beta_{0} \right\}
\]
and the inverse limit group 
\[
\varprojlim_{\beta} \pi_{m}(\sh(\Ri_{\beta}(M)))
\]
and the canonical homomorphism $\iota_{\beta,\infty}\colon \varprojlim \pi_{m}(\sh(\Ri_{\beta}(M))) \to \pi_{m}(\sh(\Ri_{\beta}(M)))$.
For a $\beta>0$, let 
$\jmath_{\beta}\colon M \to \sh(\Ri_{\beta}(M))$ be the inclusion. For each $\beta,\gamma>0$ with $\gamma < \beta$, we have $j_\beta = \iota_{\beta,\gamma}\circ j_\gamma$.
Consequently, the inclusions $\jmath_\beta\colon\pi_m (M)\to\pi_m(\sh(\Ri_\beta(M))$ induce homomorphisms
\[
\varprojlim_\beta \jmath_{\beta}\colon\pi_m (M)\to\varprojlim_\beta\pi_m(\sh(\Ri_\beta(M)).
\]

Similarly to \eqref{eq:pi_inf_S_to_M}, we define the homomorphism 
\begin{equation}\label{eq:pi_infty}
\pi_\infty\colon\varprojlim_{\beta}\pi_{m}(\sh({\Ri_{\beta}(M)})) \to \pi_{m}(M) %\text{ by }\pi_{\infty} = \pi_{\beta}\circ \iota_{\beta,\infty}.
\end{equation} 
by $\displaystyle \pi_{\infty} = \pi_{\beta}\circ \iota_{\beta,\infty}$.
In the following proposition, $\pi_{\infty}$ is established to be an isomorphism by showing that  $\varprojlim_{\beta}\jmath_\beta$ is its inverse.
%Recall that, for $\beta < \beta_0$, $\pi_{\beta}\colon \sh({\Ri_{\beta}(M)}) \to M$ denotes the restriction of the map $\pi\colon N(M) \to M$ and the homomorphism
%$\pi_{\infty}:\varprojlim_{\beta}\pi_{m}(\sh(\Ri_\beta(M))) \to M$ is defined by (\ref{1}).}

\begin{proposition}\label{prop:SbetaM}
The homomorphism 
\[
\pi_{\infty}\colon \varprojlim_{\beta}\pi_{m}(\sh({\Ri_{\beta}(M)})) \to \pi_{m}(M)
\]
is an isomorphism.
\end{proposition}

\begin{proof}
Let $\delta,\xi,\eps_\gamma$ be the parameters as defined in~\ref{assumptions}~(M2)--(M3). First, we prove the following statement:
for each pair of positive numbers $\beta, \gamma$ such that $\gamma < \beta <\beta_{0}$ and $\xi(\gamma+\varepsilon_{\gamma}) < \delta$, we have
\begin{equation}\label{systemiso}
\pi_{\beta}\circ \jmath_{\beta}= \operatorname{id}_{M},\text{ and }
\jmath_{\xi(\gamma+\varepsilon_\gamma)} \circ \pi_{\gamma} \simeq 
\iota_{\xi(\gamma+\varepsilon_\gamma),\gamma}.
\end{equation}
%where $\delta,\xi,\eps_\gamma$ are parameters as defined in~\ref{assumptions}(M2)--(M3). 

The first equality follows straightforwardly.  
For the proof of the second homotopy relation, take a point $x\in \sh(\Ri_{\gamma}(M))$ and we find finitely many points $p_{1},\ldots,p_{k}$ of $M$ such that
\[
x \in \conv(\{p_{1},\ldots,p_{k}\})~~\mbox{and}~
\diam_{M}(\{p_{1},\ldots,p_{k}\}) < \gamma.
\]
We see
\[
\begin{array}{ll}
\Vert p_{i}-p_{j}\Vert \leq d_{M}(p_{i},p_{j}) < \gamma,\\
\diam_{\mathbb{R}^N}\conv(\{p_{1},\ldots,p_{k}\}) = \max_{i,j} \Vert p_{i}-p_{j}\Vert <\gamma.
\end{array}
\]
Observe that $x \in N_{\gamma}(M)$.  By (M3), we have
\[
\Vert \pi_{\gamma}(x)-x \Vert < \varepsilon_\gamma
\]
and hence
\[
\Vert \pi_{\gamma}(x)-p_{i} \Vert \leq \Vert \pi_{\gamma}(x) - x\Vert + \Vert x- p_{i} \Vert < \varepsilon_{\gamma}+ \gamma
\]
for each $i=1,\ldots,k$.  The last term of the above is less than $\delta$ by the choice of $\gamma$. It follows from~\ref{assumptions}~(M2)~that
$d_{M}(\pi_{\gamma}(x),p_{i}) \leq \xi \Vert \pi_{\gamma}(x)-p_{i} \Vert < \xi(\gamma + \varepsilon_{\gamma}) $. 
This implies
\[
\diam_{M}(\{ \pi_{\gamma}(x),p_{1},\ldots,p_{k} \}) < \xi(\gamma+\varepsilon_{\gamma})
\]
and the points $\{ \pi_{\gamma}(x),p_{1},\ldots,p_{k} \}$ span a simplex of 
$\Ri_{\xi(\gamma+\varepsilon_{\gamma})}(M)$. We define a map
$H\colon \sh(\Ri_{\gamma}(M))\times [0,1] \to \sh(\Ri_{ \xi(\gamma+\varepsilon_{\gamma})}(M))$ by
\[
H(x,t) = tx+(1-t)\pi_{\gamma}(x) = t~ \iota_{ \xi(\gamma+\varepsilon_{\gamma}),\gamma}(x) + (1-t)~\jmath_{ \xi(\gamma+\varepsilon_{\gamma})}(\pi_{\gamma}(x)),~~(x,t) \in 
\sh(\Ri_{\gamma}(M))\times [0,1].
\]
By the above, we see that $H(x,t)$ is indeed a point of $\sh(\Ri_{ \xi(\gamma+\varepsilon_{\gamma})}(M))$ and $H$ is a well-defined homotopy between $\iota_{\ \xi(\gamma+\varepsilon_{\gamma}),\gamma}$ and $\jmath_{ \xi(\gamma+\varepsilon_{\gamma})} \circ \pi_{\gamma}$.
This proves (\ref{systemiso}).

\bigskip

For each $\beta \in (0,\beta_{0})$ we take $\gamma < \beta$ such that $\xi(\gamma+\varepsilon_{\gamma}) < \beta$.
Now we pass (\ref{systemiso}) to the homotopy groups to obtain the equalities
\begin{equation}\label{eq:systemeq}
\begin{array}{ll}
\pi_{\beta}\circ \jmath_{\beta} = \operatorname{id}_{\pi_{m}(M)}, \\
\jmath_{\beta} \circ \pi_{\gamma} 
= \iota_{\beta,\gamma}\colon \pi_{m}(\sh(\Ri_{\gamma}(M))) \to
\pi_{m}(\sh(\Ri_{\beta}(M))).
\end{array}
\end{equation}
From the above, we conclude
\[
\pi_{\infty}\circ \left(\varprojlim_{\beta}j_{\beta}\right) = \operatorname{id}_{\pi_{m}(M)},~~
\left(\varprojlim_{\beta}\jmath_{\beta}\right)
\circ \pi_{\infty} = \operatorname{id}_{\varprojlim_{\beta}\pi_m(\sh(\Ri_\beta(M)))}
\]
as follows.

For each $\omega \in \pi_{m}(M)$, we have, from~\eqref{eq:pi_infty} and the first equality of~\eqref{eq:systemeq}, that
\begin{eqnarray*}
\pi_{\infty}\circ \left(\varprojlim_{\beta}\jmath_{\beta}\right)(\omega) &=& (\pi_{\beta}\circ \iota_{\beta,\infty})\circ\big( (\jmath_{\alpha}(\omega))_{\alpha < \beta_{0}}\big )\\
&=& \pi_{\beta}(\jmath_{\beta}(\omega)) = \omega.
\end{eqnarray*}
Next, let us take $\boldsymbol{\omega} = (\omega)_{\alpha} \in \varprojlim_{\beta} \pi_{m}(\sh(\Ri_{\beta}(M)))$. For an arbitrary $\beta < \beta_0$, choose $\gamma$ so that $\xi(\gamma+\varepsilon_{\gamma}) < \beta$. We see, from the second equality of~\eqref{eq:systemeq}, that
\begin{eqnarray*}
\jmath_{\beta} (\pi_{\beta}(\omega_{\beta})) &=& \jmath_{\beta}(\pi_{\beta}(\iota_{\beta,\gamma}(\omega_{\gamma}))) \\
&=& \jmath_{\beta}\pi_{\gamma}(\omega_{\gamma}) = \iota_{\beta,\gamma}(\omega_{\gamma}) = \omega_\beta.
\end{eqnarray*}
It follows from the above and the first equality of~\eqref{eq:systemeq} that
\begin{eqnarray*}
\left(\varprojlim_{\beta}\jmath_{\beta}\right)
\circ \pi_{\infty}(\boldsymbol{\omega}) &=&  \left(\varprojlim_{\beta}\jmath_{\beta}\right) (\pi_{\beta}(\iota_{\beta,\infty}(\boldsymbol{\omega}))) \\
&=& \varprojlim_{\beta}\jmath_{\beta} (\pi_{\beta}(\omega_{\beta})) \\ &=& (\jmath_{\beta}\pi_{\beta}(\omega_{\beta})) = (\omega_{\beta}) = \boldsymbol{\omega}.
\end{eqnarray*}
This completes the proof.
\end{proof}

\begin{remark}
In terms of shape theory developed in \cite{MardesicSegal}, the above proof shows that the inverse system $\{ \sh(\Ri_{\beta}(M)), i_{\beta,\gamma} \mid o< \gamma < \beta < \beta_{0} \}$
is isomorphic to $\{ M \}$ in the category pro-HTOP.
\end{remark}

%{\color{blue}Q1. . 
%In fact, (M2-M3) may not be required in that case. 
%Would it make sense to make a remark about it here?
%k: Yes you are right. The metric we use on $M$ is not important.  This is stated in Proposition 5.2, but certainly we could mention this remark here.
%}

\begin{remark}
Proposition~\ref{prop:SbetaM} holds for more general classes of simplicial complexes. 
% $S_\beta(M)$. 
All we need to obtain the conclusion is the condition that corresponds to \eqref{eq:systemeq}. Hence if 
a system $\{S_\beta(M)\mid \beta \in \Lambda\}$ of simplicial complexes $S_\beta(M)$ indexed by a directed set $\Lambda$ such that 
%and inclusions $\iota_{\beta,\gamma}\colon S_\gamma(M) \to S_\beta(M)$ for $\gamma \geq \beta$ (in the order of $\Lambda$) so that
\begin{itemize}
    \item [(i)] for each $\beta\in \Lambda$, we have an inclusion $\jmath_\beta\colon M\to S_\beta(M)$ of $M$ into $S_\beta(M)$, and
    \item[(ii)] for $\beta,\gamma \in \Lambda$ with $\gamma\geq \beta$, we have the inclusion
    $\iota_{\beta,\gamma}\colon S_\gamma(M) \to S_\beta(M)$, and
    \item[(iii)] for each $\beta$, there exist a $\gamma > \beta$ and a map $\pi_\gamma\colon S_\gamma(M) \to M$ such that $\pi_\gamma \circ \jmath_\gamma \simeq id$, $\jmath_\beta\circ \pi_\gamma \simeq id$,  
\end{itemize}
then we have an isomorphism $\pi_{m}(M) \cong \varprojlim_{\beta}\pi_{\beta}(M)$. 
\end{remark}

%But maybe this is not what you are asking?? Maybe you are asking a generalization of Theorem 4.2 along the line of Theorem 3.3?? Then we need analogues of Lemma 4.6 and Lemma 4.7 in a general setting.

\subsection{Latchev-Type Limit Theorem for Shadow}\label{sec:Latschev}
We proceed to the proof of Theorem~\ref{thm:limitpi}.  For a finite subset $S \in \mathbb S$ which is $\beta/2$-dense in $M$, the inclusion $\jmath_{\beta}\colon M\to \sh(\Ri_{\beta}(M))$ in the proof of Proposition~\ref{prop:SbetaM} is replaced by a map $f_{\beta}^{S}\colon M\to \sh(\Ri_{\beta}(S))$ defined as follows.

For a point $s\in S$, let $D_{S}(s)$ be a closed subset of $M$ defined by 
\[
D_{S}(s) = \left\{p\in M\mid d_M(p,s)
=\min_{t\in S}d_{M}(t,p)\right\}.
\]
%{\color{blue}Why we are not allowing any points of $S$ inside $D_s(S)$ other than $s$? How do we use this property later? k: I am wondering if we do not need the property in its full strength. I need to check carefully again, but guess that all we need is "weak (9)" requiring that $\lambda_{s}^{S}(s) = 1$ and $\lambda_{s}^{S}(t) = 0$ for each $t \in S\setminus\{s\}$. I was just comfortable with using the Voronoi cell because it clarifies our geometric situation.}
We take a continuous function $\lambda_{s}^{S}\colon M\to [0,1]$ such that
\begin{equation}\label{lambda}
\lambda_{s}^{S}(s) = 1,~~~ \lambda_{s}^{S}\lvert_{M\setminus D_{S}(s)} \equiv 0.
\end{equation}
Also, let 
\[
\Lambda_{S}(p) = \sum_{s\in B_{\beta/2}(p)\cap S}\lambda_{s}^{S}(p),
\]
and observe that $\Lambda_{S}(p) > 0$ for each $p\in M$ due to the $\beta /2$-denseness of $S$ in $M$.
For $\beta>0$, we define a map $f_{\beta}^{S}\colon M\to \sh(\Ri_{\beta}(S))$ as follows.
\begin{equation}\label{fbS}
f_{\beta}^{S}(p) = \frac{1}{\Lambda_{S}(p)} \sum_{x\in B_{\beta/2}(p)\cap S} \lambda_{x}^{S}(p)\cdot x.
\end{equation}

We see that $f_{\beta}^{S}(p) \in \conv(B_{\beta/2}(p)\cap S)$ and hence $f_{\beta}^{S}$ is indeed a map to $\sh(\Ri_{\beta}(S))$.

We first prove two technical lemmas.
\begin{lemma}\label{inclusions}
For $S, S_{1}, S_{2} \in \mathbb S$ and $0< \beta, \beta_{1} < \beta_{2}$ with
$S_{1}\subset S_{2}$ and $\beta_{2} < \beta_{1}$, 
let $\iota_{\beta}^{S_{1},S_{2}}\colon \sh(\Ri_{\beta}(S_{1})) \to \sh(\Ri_{\beta}(S_{2}))$ and
$\iota_{\beta_{1},\beta_{2}}^{S}\colon \sh(\Ri_{\beta_2}(S)) \to \sh(\Ri_{\beta_1}(S))$ be the inclusions. 
We have the following:
\begin{enumerate}
\item $\iota_{\beta}^{S_{1},S_{2}} \circ f_{\beta}^{S_{1}} \simeq f_{\beta}^{S_{2}}.$
\item $\iota_{\beta_{1},\beta_{2}}^{S} \circ f_{\beta_2}^{S} \simeq f_{\beta_1}^{S}.$
\end{enumerate}
\end{lemma}

\begin{proof}
(1).  For $\beta>0$ and $S_{1}, S_{2} \in \mathbb S$ with $S_{1} \subset S_{2}$, take a point $p$ of $M$.  By the definition of $f_{\beta}^{S_1}$ and $f_{\beta}^{S_2}$ we have
\[
 f_{\beta}^{S_1}(p) \in \conv(S_{1}\cap B_{\beta/2}(p)),~\mbox{and}~~
 f_{\beta}^{S_2}(p) \in \conv(S_{2}\cap B_{\beta/2}(p)).
\]
By the assumption, $\conv(S_{2}\cap B_{\beta/2}(p)) \supset
\conv(S_{1}\cap B_{\beta/2}(p))$, thus we see for each $t\in [0,1]$,
\[
(1-t) ~\iota_{\beta}^{S_{1},S_{2}} (f_{\beta}^{S_{1}}(p)) + t~ f_{\beta}^{S_{2}}(p) 
\in \conv(S_{2}\cap B_{\beta/2}(p)) \subset
\sh(\Ri_\beta(S_{2})).\]
Thus the map $M\times [0,1] \to \sh(\Ri_\beta(S_{2}))$ defined by 
\[
t\mapsto (1-t) \iota_{\beta}^{S_{1},S_{2}} \circ f_{\beta}^{S_{1}} + t f_{\beta}^{S_{2}} 
\]
gives the desired homotopy between $\iota_{\beta}^{S_{1},S_{2}} \circ f_{\beta}^{S_{1}}$ and $f_{\beta}^{S_{2}}$.
This proves (1).  The proof of (2) is similar to the above.

\end{proof}

\begin{lemma}\label{pifbs}
Let $\beta$ be a positive number satisfying
\[
2\beta+\epsilon_{\beta} < \delta,~~~\xi(2\beta+\varepsilon_{\beta}) < \rho(M), 
\]
and define $\nu_\beta$ by
\[
\nu_{\beta} = ((1/2)+\xi)\beta+ \xi\varepsilon_{\beta}.
\]
% ($\beta$-1),($\beta$-2)  and ($\beta$-3) 
Then, for each $S\in \mathbb S$ which is $\beta/2$-dense in $M$, we have the following.:
\begin{enumerate}
\item $\pi_{\beta} \circ f_{\beta}^{S} \simeq \operatorname{id}_{M}$.
\item %Let $\nu_{\beta}= (2\xi + \frac{1}{2})\beta$. Then 
$\displaystyle
\iota_{\nu_{\beta},\beta}^S\circ f_{\beta}^{S} \circ \pi_{\beta} \simeq \iota_{\nu_{\beta},\beta}^{S}\colon \sh(\Ri_{\beta}(S)) \to \sh(\Ri_{\nu_\beta}(S)).
$
\end{enumerate}
\end{lemma}

\begin{proof} 
Take $\beta$ and $S\in \mathbb S$ as in the hypothesis.

(1)  For a point $p$ of $M$, we have $f_{\beta}^{S}(p) \in \conv(S \cap B_{\beta/2}(p))$. Also we see
\[
\diam_{\mathbb{R}^{N}}(S \cap B_{\beta/2}(p)) \leq \diam_{M}(S \cap B_{\beta/2}(p)) < \beta.
\]
We observe from the above that $f_{\beta}^{S}(p) \in N_{\beta}(M)$. Hence, by~\ref{assumptions}~(M3),  we have
\begin{equation}\label{norm1}
\Vert \pi_{\beta}(f_{\beta}^{S}(p))-f_{\beta}^{S}(p)\Vert < \varepsilon_\beta.
\end{equation}
Also for each point $x\in B_{\beta/2}(p)\cap S$, we have 
\[
\Vert x-p\Vert \leq d_{M}(x,p) < \beta/2.
\]
Since $f_{\beta}^{S}(p) \in \conv(S\cap B_{\beta/2}(p))$, we see from the above that 
\[
\Vert f_{\beta}^{S}(p) - p \Vert < \beta/2.
\]		
These two imply
\begin{eqnarray*}
\Vert p-\pi_{\beta}(f_{\beta}^{S}(p))\Vert &\leq& \Vert p-f_{\beta}^{S}(p) \Vert + \Vert f_{\beta}^{S}(p)-\pi_{\beta}(f_{\beta}^{S}(p))\Vert \\
&<& (\beta/2) +\varepsilon_{\beta}.
\end{eqnarray*}
The last term is less than $\delta$ by the choice of $\beta$. Hence, we obtain by~\ref{assumptions}~(M2) that
\[
d_{M}(p,\pi_{\beta}(f_{\beta}^{S}(p))) < \xi ((\beta/2) + \varepsilon_\beta) < \rho(M).
\]
Since $p$ is an arbitrary point of $M$, we obtain that 
$\pi_{\beta}\circ f_{\beta}^{S}$ is $\rho(M)$-close to $\operatorname{id}_M$.
From~\ref{assumptions}~(M1) , we obtain $\pi_{\beta}\circ f_{\beta}^{S} \simeq \operatorname{id}_M$.

(2).  For a point $x\in \sh(\Ri_{\beta}(S))$, there are points $p_{1},\ldots,p_{k}$ of $M$ such that
\[
x\in \conv(\{p_{1},\ldots,p_{k}\}),~~~ \diam_{M}(\{p_{1},\ldots,p_{k}\}) < \beta.
\]
We observe that $x\in N_{\beta}(M)$ and $\diam_{\mathbb{R}^N}(\{p_{1},\ldots,p_{k}\}) < \beta$. 
For each $i=1,\ldots,k$, we have, from~\ref{assumptions}~(M3), 
\begin{eqnarray*}
\Vert \pi_{\beta}(x) -p_{i}\Vert &\leq& \Vert \pi_{\beta}(x)-x\Vert + \Vert x-p_{i}\Vert \\
&\leq&  \varepsilon_{\beta} +  \beta.
\end{eqnarray*}
The last term is less than $\delta$ and by~\ref{assumptions}~(M2), we see
\[
d_{M}(\pi_{\beta}(x),p_{i}) < \xi(\beta+\varepsilon_{\beta}),~~~i=1,\ldots,k.
\]
From the above, it follows that for each $y\in S\cap B_{\beta/2}(\pi_{\beta}(x))$ and for each $i=1,\ldots k$,
\begin{eqnarray*}
d_{M}(y,p_{i}) &\leq& d_{M}(y,\pi_{\beta}(x))+d_{M}(\pi_{\beta}(x),p_{i}) \\
&<& (\beta/2) + \xi(\beta+\varepsilon_{\beta}) = ((1/2)+\xi)\beta+ \xi\varepsilon_{\beta} = \nu_{\beta}, 
\end{eqnarray*}
which implies
\[
\diam_{M}(\{p_{1},\ldots,p_{k}\}\cup (S\cap B_{\beta/2}(p))) < \nu_{\beta},
\]
Hence for each $t\in [0,1]$, 
\[
(1-t)x+ t f_{\beta}^{S}(\pi_{\beta}(x)) \in \conv(\{p_{1},\ldots,p_{k}\}\cup (S\cap B_{\beta/2}(\pi_{\beta}(x)))) \subset \sh(\Ri_{\nu_{\beta}}(S)).
\]
The map $\sh(\Ri_{\beta}(S))\times [0,1] \to \sh(\Ri_{\nu_\beta}(S))$ defined by $(x,t) \mapsto (1-t)x+ t f_{\beta}^{S}(\pi_{\beta}(x))$ gives a homotopy between the maps $\iota_{\nu_{\beta},\beta}^{S}$ and $\iota_{\nu_{\beta},\beta}^{S}\circ f_{\beta}^{S} \circ \pi_{\beta}$.  This proves (2).
  
\end{proof}

We finally conclude this section by proving Theorem~\ref{thm:limitpi}. 
\begin{proof}[Proof of Theorem~\ref{thm:limitpi}]
We apply Lemma \ref{pifbs} to obtain the commutative diagrams of homotopy groups:

\[
\xymatrix{
  \pi_{m}(M) \ar[r]^{f_{\beta}^{S}} \ar[rd]_{\operatorname{id}} & \pi_m(\sh(\Ri_\beta (S))) \ar[d]^{\pi_{\beta}^{S}}
\\
& \pi_{m}(M)
}
\]
and
\[
\xymatrix{
  \pi_{m}(\sh(\Ri_\beta (S)))  \ar[r]^{\pi_{\beta}^{S} } \ar[rd]_{\operatorname{id}} & \pi_{m}(M) \ar[rd]^{\iota_{\nu_{\beta},\beta}^{S}\circ f_{\beta}^{S}
}
\\
& \pi_{m}(\sh(\Ri_{\beta}(S))) \ar[r]_{\iota_{\nu_{\beta},\beta}^{S}} &\pi_{m}(\sh(\Ri_{\nu_\beta}(S)))
}
\]
for each sufficiently small $\beta < \beta_0$ satisfying the hypothesis of the lemma.
We now take the direct limit $\varinjlim_{S\in \mathbb S}$ to obtain the corresponding commutative diagrams:
\begin{equation}\label{explanation1}
\xymatrix{
  \pi_{m}(M) \ar[r]^{f_{\beta}^{\mathbb S}} \ar[rd]_{\operatorname{id}} & \sh(\Ri_\beta (\mathbb S)) \ar[d]^{\pi_{\beta}^{\mathbb S}}
\\
& \pi_{m}(M)
}
\end{equation}
and
\begin{equation}\label{explanation2}
\xymatrix{
  \pi_{m}(\sh(\Ri_\beta (\mathbb S))  \ar[r]^{\pi_{\beta}^{\mathbb S} } \ar[rd]_{\operatorname{id}} & \pi_{m}(M) \ar[rd]^{\iota_{\nu_{\beta},\beta}^{\mathbb S}\circ f_{\beta}^{\mathbb S}
}
\\
& \pi_{m}(\sh(\Ri_{\beta}(\mathbb S))) \ar[r]_{\iota_{\nu_{\beta},\beta}^{\mathbb S}} &\pi_{m}(\sh(\Ri_{\nu_\beta}(\mathbb S)))
}
\end{equation}
Then, we take the inverse limit $\varprojlim_{\beta}$ to obtain
\begin{equation}\label{explanation3}
\xymatrix{
  \pi_{m}(M) \ar[r]^{\varprojlim_{\beta} f_{\beta}^{\mathbb S}\quad\quad\quad} \ar[rd]_{\operatorname{id}} & \varprojlim_{\beta}\pi_{m}(\sh(\Ri_\beta (\mathbb{S}))) \ar[d]^{\pi_{\infty}^{\mathbb S}}
\\
& \pi_{m}(M)
}
\end{equation}
and
\begin{equation}\label{explanation4}
\xymatrix{
  \varprojlim_{\beta} \pi_{m}(\sh(\Ri_\beta (\mathbb{S})) ) \ar[r]^{\pi_{\infty}^{\mathbb S} } \ar[rd]_{\operatorname{id}} & \pi_{m}(M) \ar[d]^{\varprojlim_{\beta}f_{\beta}^{\mathbb S}
}
\\
& \varprojlim_{\beta} \pi_{m}(\sh(\Ri_{\beta}(\mathbb{S})) )
}
\end{equation}

Here we verify the commutativity of (\ref{explanation3}) and (\ref{explanation4}) as follows: 

For each $\omega \in \pi_{m}(M)$, we have from (\ref{explanation1})
\begin{eqnarray*}
\pi_{\infty}^{\mathbb S} (\varprojlim_{\beta}f_{\beta}^{\mathbb S}(\omega)) &=& 
(\pi_{\beta}^{\mathbb S}\circ \iota_{\beta,\infty}) \left (\varprojlim_{\beta}f_{\beta}^{\mathbb S}(\omega)\right) \\
&=& \pi_{\beta}^{\mathbb S}(f_{\beta}^{\mathbb S}(\omega)) = \omega.
\end{eqnarray*}
This verifies the commutativity in (\ref{explanation3}).
For (\ref{explanation4}), take an arbitrary $\beta < \beta_0$ and choose $\gamma<\beta$ so that 
$\nu_{\gamma} = \beta$. Applying the commutativity in (\ref{explanation2}) to $\gamma$, we see
\[
\iota_{\beta,\gamma}^{\mathbb S}\circ f_{\gamma}^{\mathbb S}\circ \pi_{\gamma}^{\mathbb S} = \iota_{\beta,\gamma}^{\mathbb S}.
\]
From the above, we obtain 
\begin{eqnarray*}
\iota_{\beta,\infty}^{\mathbb S}\circ \varprojlim_{\beta} f_{\beta}^{\mathbb S}\circ \pi_{\infty}^{\mathbb S} 
&=& \iota_{\beta,\gamma}^{\mathbb S}\circ (\iota_{\gamma,\infty}^{\mathbb S} 
\circ \varprojlim_{\beta} f_{\beta}^{\mathbb S} )\circ 
\pi_{\infty}^{\mathbb S}\\
&=& \iota_{\beta,\gamma}^{\mathbb S} \circ f_{\gamma}^{\mathbb S} 
\circ  \pi_{\infty}^{\mathbb S} \\
&=&  (\iota_{\beta,\gamma}^{\mathbb S} \circ f_{\gamma}^{\mathbb S} 
\circ \pi_{\gamma}^{\mathbb S}) \circ \iota_{\gamma,\infty}^{\mathbb S}
\\
&=& \iota_{\beta,\gamma}^{\mathbb S}\circ \iota_{\gamma,\infty}^{\mathbb S} =
\iota_{\beta,\infty}^{\mathbb S}.
\end{eqnarray*}
Since the above holds for arbitrary $\beta>0$, we see that (\ref{explanation4}) is commutative.

This shows that $\pi_{\infty}^{\mathbb S}$ is an isomorphism, whose inverse is given by $\varprojlim_{\beta}f_{\beta}^{\mathbb S}$.
This completes the proof.
\end{proof}

%\vspace{5mm}

\subsection{Limit Theorem for Shadow Projection}
We make use of the above result to study the homotopy behavior of the shadow projection map $p_{\beta}^{S}\colon \Ri_\beta(S) \to \sh(\Ri_\beta(S))$ of the Vietoris-Rips complex $\Ri_\beta(S)$ of a sample $S\subset M$ at scale $\beta>0$.  
We first recall Hausmann's construction~\cite{hausmann1995vietoris}. 

Throughout we fix a total order on $M$ and let $\sigma = [p_{0},\ldots,p_{n}]$ be a simplex of $\Ri_\beta(M)$, i.e., $\{p_{0},\ldots, p_{n}\} \subset M$ and $\diam_{M}(\{p_{0},\ldots, p_{n}\}) < \beta$.
We assume that the vertices are enumerated as $p_{0} < \cdots < p_{n}$.

The Hausmann map $T\colon  \Ri_\beta (M) \to M$ is defined inductively on the dimension $n\geq0$. 
First set $T(p) = p$.  
Assume that $T$ is defined on the $(n-1)$-skeleton $\Ri_\beta (M)^{(n-1)}$ of $\Ri_\beta(M)$ and let $\sigma =[p_{0},\ldots, p_{n}]$ be an $n$-simplex. For a point $\displaystyle x = \sum_{i=0}^{n}\lambda_{i}p_{i}$ of $\Ri_\beta(M)$ (where $\lambda_{i}\geq 0, \sum_{i}\lambda_{i} =1$, $T(x)$ is defined by
\begin{equation}\label{hausmannmap}
T(x) =
\begin{cases}
p_{n},~~ \mbox{if}~ \lambda_{n} = 1,\\
T\left (\frac{1}{1-\lambda_n}\sum_{i=0}^{n-1}\lambda_{i}p_{i}\right),~~~\mbox{if}~\lambda_{n} < 1.
\end{cases}
\end{equation} 
When the metric $d_M$ satisfies the following conditions (\cite[p.179, Items b)-c)]{hausmann1995vietoris}), 
the diameter of the image $T(\sigma)$ of an arbitrary simplex $\sigma=[p_{0},\ldots,p_{n}]$ of $\Ri_\beta(M)$ is estimated in the next lemma.
%The image $T(\sigma)$ of an arbitrary simplex $\sigma=[p_{0},\ldots,p_{n}]$ of $\Ri_\beta(M)$ satisfies the following when the metric $d_M$ satisfies the following conditions (\cite[Items b)-c)]{hausmann1995vietoris}).
\begin{itemize}
    \item [(i)] Let $p,q,r$ be points of $M$ such that $\max\{ d_{M}(p,q),d_{M}(q,r),d_{M}(r,p)\} < \rho(M)$ and let $s$ be a point on a geodesic joining $p$ and $q$. Then we have
    \[
    d_{M}(r,s) \leq \max\{d_{M}(r,p),d_{M}(r,q)\}
    \]
    \item[(ii)] If $c_{1},c_{2}$ are two geodesics such that $c_{1}(0) = c_{2}(0)$ and if $s_{1},s_{2} \in [0,\rho(M)]$, then we have
    \[
    d_{M}(c_{1}(ts_{1}),c_{2}(ts_{2})) \leq d_{M}(c_{1}(s_{1}),c_{2}(s_{2})).
    \]
\end{itemize}
\begin{lemma}\label{diamhausmann}
Assume that $M$ satisfies the conditions (i) and (ii) above and let $\beta$ be a positive number so that $2\beta <\rho(M)$. Then, we have the following:
\begin{itemize}
\item[(1)] For each point $q\in M$ with $d_{M}(q,p_{i}) < \beta,~i=0,\ldots,n$, and for each  point $r\in T(\sigma)$, we have $d_{M}(r,q) < \beta.$
\item[(2)] For each point $q\in T(\sigma)$, we have $d_{M}(q,p_{i}) < \beta,~i=0,\ldots,n$.
\end{itemize}
\end{lemma} 
\begin{proof}
Both proofs consider induction on $n$.

(1) The proof for the case  $n=1$ is straightforward.  Assume that the statement holds for $(n-1)$ and take an $n$-simplex $[p_{0}, \ldots,p_{n}]$ and points $q, r$ as in the hypothesis. We may assume that $r\neq p_{n}$. By the construction (\ref{hausmannmap}) there exists a point $r' \in T([p_{0}, \ldots,p_{n-1}])$ such that $r$ lies on the unique geodesic $c_{p_{n}r'}$ between $p_n$ and $r'$. By the inductive hypothesis, we have $d_{M}(q,r') < \beta$. 
Applying Condition (i) above, we have 
%\cite[p.179, item b)]{hausmann1995vietoris}{\color{blue}[Here we should explicitly state items a)-c).  I will add them later]}, we have
\[
d_{M}(r,q) \leq \max\{ d_{M}(q,p_{n}), d_{M}(q,r') \} < \beta.
\]

(2) The case $n=1$ is a direct consequence of  Condition (i). Assuming the conclusion holds for $(n-1)$, we take an $n$-simplex  $\sigma = [p_{0}, \ldots,p_{n}]$. For each point $q\in T(\sigma) \setminus \{p_{n}\}$ there exists a point $q' \in T([p_{0}, \ldots,p_{n-1}])$ such that $q$ is on the unique geodesic $c_{p_{n}q'}$ joining $p_n$ and $q'$. It follows from (1) that $d_{M}(q',p_{n}) < \beta$.  Also by the inductive hypothesis, $d_{M}(p_{i},q') < \beta$ for each $i=0,\ldots,n-1$.
Hence we obtain for $i=0,\ldots, (n-1)$,
\[
d_{M}(q, p_{i}) \leq \max\{ d_{M}(p_{i},p_{n}), d_{M}(p_{i},q') \} < \beta.
\]
This proves (2).
\end{proof}

The above observation motivates the following assumption:

\bigskip
\noindent
{\bf Assumption (H)} .~~
There exists $\beta_{0}>0$ and a homotopy equivalence
\[
T\colon \Ri_{\beta_{0}}(M) \to M
\]
such that for each $\beta \in (0,\beta_{0})$ and for each $\sigma = [p_{0},\ldots,p_{n}] \in \Ri_{\beta}(M)$ we have
\[
T(\sigma) \subset \bigcap_{i=0}^{n}B_{d_M}(p_{i},\beta).
\]
For $\beta < \beta_0$, let $T_{\beta}\coloneq  T\lvert_{\Ri_\beta(M)}$ and it is called a \emph{Hausmann map}.
We start by comparing the Vietoris-Rips shadow projection map
$p_{\beta}^{S}\colon \Ri_\beta(S) \to \sh(\Ri_\beta(S))$ with a Hausmann map $T_{\beta}\colon \Ri_\beta(M) \to M$.

\begin{proposition}\label{projectionhausmann}
Assume that $\beta>0$ satisfies 
\[
\beta+\eps_\beta<\delta\text{ and }\beta+\xi(\beta+\varepsilon_{\beta}) < \rho(M).
\]
%$(\beta-1), (\beta-2)$ and $(\beta-3)$.
For a closed subset $F$ of $M$, let $\iota_{\beta}^{F}\colon \Ri_\beta(F) \to \Ri_\beta(M)$ and $\pi_{\beta}^{F}\colon \sh(\Ri_{\beta}(F)) \to M$ be the inclusions and the projection respectively. 
Then, we have
\[
\pi_{\beta}^{F} \circ p_{\beta}^{F} \simeq T_{\beta} \circ \iota_{\beta}^{F},
\]
i.e., the following diagram is commutative up to homotopy:
\begin{equation}\label{projectiondiagram-1}
\xymatrix{
  \Ri_{\beta}(F) \ar[r]^{p_{\beta}^{F}} \ar[d]_{\iota_{\beta}^F} & \sh(\Ri_\beta (F)) \ar[d]^{\pi_{\beta}^{F}}
\\
\Ri_\beta(M) \ar[r]_{T_\beta} & M
}
\end{equation}
\end{proposition}

\begin{proof}

Recalling that the complex $\Ri_\beta(F)$ is endowed with Whitehead topology, that is, the weak topology with respect to the set of all simplices, we see from the Claim in the proof of Proposition~\ref{prop:VRlimit} that every compact subset of $\Ri_\beta(M)$ is contained in a finite subcomplex.  
Thus, it suffices to prove the above for each finite subset $F$ of $M$.

Let $T = T_\beta$ for simplicity. 
Take a point $x\in \Ri_\beta(F)$ and choose a simplex $\sigma=[p_{0},\ldots,p_{n}]$ of $\Ri_\beta(F)$:  $\{p_{0},\ldots, p_{n}\} \subset F$, $d_{M}(p_{i},p_{j}) < \beta,~ i,j = 0,\ldots, n$, such that $\displaystyle x= \sum_{i=0}^{n}\lambda_{i}p_{i}$.
We have $T(\iota_{\beta}^{F}(x)) \in T(\sigma)$ and by Lemma \ref{diamhausmann} (2), 
\begin{equation}\label{ipi1m}
d_{M}(T(\iota_{\beta}^{F}(x)),p_{i}) < \beta,~~ i=0,\ldots,n.
\end{equation}
On the other hand, $p_{\beta}^{F}(x) \in \conv(\{p_{0},\ldots,p_{n}\})$ and 
\[
\diam_{\mathbb{R}^N}(\conv(\{p_{0},\ldots,p_{n}\}) \leq \diam_{M}(\conv(\{p_{0},\ldots,p_{n}) < \beta,
\]
which implies $p_{\beta}^{S}(x) \in N_{\beta}(M)$.  It follows from~\ref{assumptions}~(M3) 
\[
\Vert \pi_{\beta}^{F}(p_{\beta}^{F}(x)) - p_{\beta}^{F}(x)\Vert 
< \varepsilon_\beta.
\]
From this, we obtain
\[
\Vert \pi_{\beta}^{F}(p_{\beta}^{F}(x)) - p_{i}\Vert \leq 
\Vert \pi_{\beta}^{F}(p_{\beta}^{F}(x)) - p_{\beta}^{F}(x) \Vert + 
\Vert p_{\beta}^{F}(x) - p_{i} \Vert < \varepsilon_{\beta}+\beta.
\]
From~\ref{assumptions}~(M2) it follows that
\begin{equation}\label{ipi2m}
d_{M}(\pi_{\beta}^{F}(x),p_{i}) < \xi (\beta+\varepsilon_{\beta}),~~i=0,\ldots, n.
\end{equation}
From (\ref{ipi1m}) and (\ref{ipi2m}), we have
\[
d_{M}(T(\iota_{\beta}^{F}(x)),\pi_{\beta}^{F}(p_{\beta}^{F}(x))) < \beta+\xi(\beta+\varepsilon_{\beta}) < \rho(M).
\]
Thus, the maps $T\circ \iota_{\beta}^F$ and $\pi^{F}_{\beta}\circ p_{\beta}^{F}$ are $\rho(M)$-close and hence they are homotpic.
This completes the proof.
\end{proof}

%Applying Proposition \ref{projectionhausmann} to a sample set which is sufficiently dense in $M$ and using Proposition \ref{onedim}, we immediately obtain the following.
%
%\begin{proposition}
%Let $M$ be a one-dimensional compact connected submanifold of $\mathbb{R}^N$.
%Assume that $\beta>0$ satisfies $(\beta$-1), ($\beta$-2) and ($\beta$-3) and let $S$ be either $M$ or a finite subset of $M$ which is $\beta/2$-dense in $M$. Then the Vietoris-Rips projection map $p_{\beta}^{S}\colon \Ri_\beta(S) \to \sh(\Ri_\beta(S))$ is a homotopy equivalence.
%\end{proposition}

%Next we proceed to prove an analogue of the above for higher dimensional submanifolds by applying Theorem \ref{limitpi}. 
For a finite set $S \in \mathbb S$, the Vietoris-Rips shadow projection $p_{\beta}^{S}\colon \Ri_\beta(S) \to \sh(\Ri_\beta(S))$ induces homomorphism on the $m$-homotopy groups:
\[
p_{\beta}^{S}\colon \pi_{m}(\Ri_\beta(S)) \to \pi_{m}(\sh(\Ri_\beta(S))),
\]
which induces a homomorphism on direct limits: 
\[
p_{\beta}^{\mathbb S}\colon \pi_{m}(\Ri_\beta(\mathbb{S})) \to \pi_{m}(\sh(\Ri_\beta(\mathbb{S}))).
\]
Taking the inverse limit with respect to $\beta$, we obtain the homomorphism of the following theorem.

\begin{theorem}\label{limitproj}
The inverse limit homomorphism
\[
\varprojlim_{\beta} p_{\beta}^{\mathbb S}\colon  \varprojlim_{\beta} \pi_{m}(\Ri_\beta(\mathbb{S})) \to \varprojlim_{\beta}\pi_{m}(\sh(\Ri_\beta(\mathbb{S})))
\]
is an isomorphism for each $m\geq0$.

\end{theorem}

\begin{proof}
We start with taking the direct limit $\varinjlim_{S}$ in the diagram (\ref{projectiondiagram-1}) to obtain a commutative diagram:

\begin{equation}\label{projectiondiagram}
\xymatrix{
  \pi_{m}(\Ri_{\beta}(\mathbb{S})) \ar[r]^{p_{\beta}^{\mathbb S}} \ar[d]_{\iota_{\beta}^{\mathbb S}} & \pi_{m}(\sh(\Ri_\beta (\mathbb{S}))) \ar[d]^{\pi_{\beta}^{\mathbb S}}
\\
\pi_{m}(\Ri_\beta(M)) \ar[r]_{T} & \pi_{m}(M)
}
\end{equation}
Recalling that the complex $\Ri_\beta(M)$ is endowed with Whitehead topology, we see that the homomorphism
\[
\iota_{\beta}^{\mathbb S}\colon \pi_{m}(\Ri_\beta(\mathbb{S})) \to \pi_{m}(\Ri_\beta(M))
\]
is an isomorphism; see Remark~\ref{rem:VR1}. 
Also, $T$ is an isomorphism by \cite{hausmann1995vietoris}.  
Hence we see that the homomorphism $\pi_{\beta}^{\mathbb S}\circ p_{\beta}^{\mathbb S}$ is an isomorphism.  
Taking the inverse limit, we see
\[
\pi_{\infty}^{\mathbb S} \circ \varprojlim_{\beta}p_{\beta}^{\mathbb S} ~\mbox{is an isomorphism}.
\]
Now by Theorem~\ref{thm:limitpi}, $\pi_{\infty}^{\mathbb S}$ is an isomorphism, and hence so is $\varprojlim_{\beta}p_{\beta}^{\mathbb S}$.
This proves the theorem.
\end{proof}

\section{Limit theorems for noisy samples}\label{sec:shadowLimitNoisy}

Thus far, we have considered only finite sample sets $S$ that lie directly on $M$.
In this section, we study the case where samples are taken from a neighborhood $N_\tau(M)$ of the compact subset $M$ of $\mathbb{R}^N$.

From the shape reconstruction \cite{SushThesis} viewpoint, it is natural to examine the {\it Euclidean} Vietoris-Rips complex $\Ri^{\mathbb{R}^N}_\beta(S)$ and its shadow $\sh(\Ri^{\mathbb{R}^N}_\beta(S))$ of a sample set $S \subset N_{\tau}(M)$ equipped with the Euclidean metric.  
%They are defined by means of Euclidean distances of points of $S$.  
On the other hand, another metric {\it the $\varepsilon$-path metric}  on the sample set $S$ was introduced \cite{fasy2022reconstruction,majhi2023vietoris,MajhiStability} in the reconstruction context to obtain homotopy equivalence $\Ri_\beta(S) \simeq M$ for any sufficiently small $\beta$ and for any sample set $S$ sufficiently close to $M$ in the Hausdorff distance---when $M$ is a Euclidean-embedded graph~\cite{majhi2023vietoris} or a CAT($\kappa$) space of $\mathbb R^{N}$~\cite{MajhiStability}.  
For sufficiently small $\beta,\tau>0$, the $\varepsilon$-path metric $d^\varepsilon$ has the bounded local distortion with respect to the Euclidean distance.  Our setup below is slightly more general than that for the metric $d^{\varepsilon}$.

{\bf Assumption (N)}.  ~~
We assume that the neighborhood $N(M)$ of the conditions (M1-M3) admits a metric $d_0$ such that
\begin{itemize}
\item[(N)] there exists $\delta_{0}>0$ and $\kappa_{1},\kappa_{2} >1$ such that for each pair of points $p,q$ of $N(M)$ with $\max\left\{\Vert p- q\Vert, d_{0}(p,q)\right\}  < \delta_0$, we have
\[
\kappa_{1}^{-1}d_{0}(p,q) \leq \Vert p-q\Vert \leq \kappa_{2}d_{0}(p,q).
\]
\end{itemize}
For $\tau>0$ with $N_{\tau}(M) \subset N(M)$, $d_\tau$ denotes the restriction of the metric $d_0$ on $N(M)$.

\begin{remark}
Note for the $\varepsilon$-path metric $d^{\varepsilon}$, we have
\[
\Vert p-q \Vert < \varepsilon ~~\Rightarrow d^{\varepsilon}(p,q) = \Vert p-q \Vert.
\]
This implies
\[
\beta< \varepsilon \Rightarrow \Ri_{\beta}^{\mathbb R^{N}}(N_{\tau}(M)) = \Ri_{\beta}^{d_\tau}(N_{\tau}(M)).
\]
\end{remark}

\bigskip

In what follows, the set $\{ (\beta,\tau)\mid \beta, \tau >0 \}$ is regarded as a directed set by the order
\[
(\beta_{1},\tau_{1}) \geq (\beta_{2},\tau_{2})~~\Leftrightarrow \beta_{1} \leq \beta_{2}~\mbox{and}~\tau_{1}\leq \tau_{2}.
\]

\bigskip

A natural question is whether our limit results for noisy samples depend on the choice of metrics on $N_{\tau}(M)$.
The next proposition answers the question. For $\beta>0$ with $\kappa_{1}\kappa_{2} \beta < \delta_0$, $S_{1},S_{2}\in \mathbb{S}_\tau$, we have the following inclusions from Assumption (N):
\[
\xymatrix{
  \sh(\Ri_{\beta}^{d_\tau} (S_{1})) \ar[r]^{j_{\beta}^{S_1}} \ar[d] & 
  \sh(\Ri_{\kappa_{2}\beta}^{\mathbb{R}^N} (S_{1})) \ar[r]^{k_{\beta}^{S_1}}  \ar[d] &
  \sh(\Ri_{\kappa_{1}\kappa_{2}\beta}^{d_\tau}(S_{1})) \ar[d]
\\
\sh(\Ri_{\beta}^{d_\tau}(S_{2})) \ar[r]_{j_{\beta}^{S_2}}  & 
  \sh(\Ri_{\kappa_{2}\beta}^{\mathbb{R}^N} (S_{2})) \ar[r]_{k_{\beta}^{S_2}}  &
  \sh(\Ri_{\kappa_{1}\kappa_{2}\beta}^{d_\tau}(S_{2})).
}
\]
Where the vertical arrows also represent appropriate inclusions. For $\beta,\gamma$ with
$0< \kappa_{1}\kappa_{2}\gamma < \beta$, we take the direct limits of the corresponding homotopy groups and obtain: 
\[
k_{\kappa_{2}\gamma}^{\mathbb{S_{\tau}}} \circ j_{\gamma}^{\mathbb{S}_\tau} = \iota_{\beta,\gamma}^{d_{\tau}}
\]
and
\[
j_{\beta}^{\mathbb S_\tau} \circ k_{\gamma}^{\mathbb S_\tau}= \iota_{\beta,\gamma}^{\mathbb R^N}.
\]
where 
$\iota_{\beta,\gamma}^{d_\tau}:\pi_{m}(\sh(\Ri_{\gamma}^{d_\tau}(\mathbb S_{\tau})) \to 
\pi_{m}(\sh(\Ri_{\beta}^{d_\tau}(\mathbb S_{\tau})))$ and
$\iota_{\beta,\gamma}^{\mathbb R^N}:\pi_{m}(\sh(\Ri_{\gamma}^{d_\tau}(\mathbb S_{\tau})) \to 
\pi_{m}(\sh(\Ri_{\beta}^{\mathbb R^N}(\mathbb S_{\tau})))$ 
are homomorphisms induced by inclusions.

%we have the following commutative diagram for each $\beta,\gamma$ with $\gamma < (\kappa_{1}\kappa_{2})^{-1}\beta$
%\[
%\xymatrix{
% \pi_{m}(\sh(\Ri_{\beta}^{d_\tau}(\mathbb{S}_{\tau}))) \ar[d]_{j_{\beta}^{\mathbb{S}_\tau}} &  \pi_{m}(\sh(\Ri_{\gamma}^{d_\tau}(\mathbb{S}_{\tau}))) \ar[d]^{j_{\gamma}^{\mathbb{S}_\tau}} \ar[l]_{i_{\beta,\gamma}} \\
%\pi_{m}(\sh(\Ri_\beta^{\mathbb{R}^N} (\mathbb{S}_{\tau}))) % 
%&\pi_{m}(\sh(\Ri_\gamma^{\mathbb{R}^N} (\mathbb{S}_{\tau}))) \ar[l]^{i_{\beta,\gamma}} \ar[lu]_{k_{\beta}^{\mathbb{S}_\tau}}
%}
%\]
%where two $i_{\beta,\gamma}$'s are homomorphisms induced by inclusions. 
The same holds for the Vietoris-Rips complexes.  From these, we conclude the following proposition.
\begin{proposition}\label{EuclideanM}
The inverse limit homomorphisms
\[
\begin{array}{ll}
\varprojlim_{(\beta,\tau)} j_{\beta}^{\mathbb{S}_{\tau}}\colon  \varprojlim_{(\beta,\tau)} \pi_{m}(\Ri_{\beta}^{d_\tau} (\mathbb{S}_{\tau})) \to \varprojlim_{(\beta,\tau)}\pi_{m}(\Ri_\beta^{\mathbb{R}^N} (\mathbb{S}_{\tau})),\\ 
\varprojlim_{(\beta,\tau)} j_{\beta}^{\mathbb{S}_{\tau}}\colon  \varprojlim_{(\beta,\tau)} \pi_{m}(\sh(\Ri_{\beta}^{d_\tau} (\mathbb{S}_{\tau}))) \to \varprojlim_{(\beta,\tau)}\pi_{m}(\sh(\Ri_\beta^{\mathbb{R}^N} (\mathbb{S}_{\tau}))) 
\end{array}
\]
are isomorphisms.
\end{proposition}

\begin{remark}
The above result allows us to consider $\Ri_\beta(S)$ and its shadow for either of the metrics $d^\tau$ and $\Vert \cdot -\cdot \Vert$, as long as we are interested only in the limit results.  It has been pointed out in \cite{majhi2023vietoris} that for an {\it fixed} finite sample $S$ of a neighborhood of a metric graph $\G$, $\Ri_{\beta}^{\mathbb R^{N}}(S)$ may not be homotopy equivalent to $\G$. Such subtlety disappears in the limit.
\end{remark}

%in which the set of pairs $(\beta,\tau)$ satisfying ($\beta\tau$-1), ($\beta\tau$-2) and ($\beta\tau$-3) forms a cofinal subset.

%\bigskip

%{\color{blue} (*)  
%In the modified set up in which we do not assume that $M$ is a manifold, we will use the Euclidean metric for $N_\tau(M)$ from the beginning. This will simplify a few arguments.}

\subsection{Hausmann-Type Limit Theorem for Shadow under Noise}
Proposition~\ref{prop:SbetaM} has the following analogue.
\begin{proposition}\label{SbetaMnbd}
Let $\Ri_{\beta}(N_{\tau}(M)) = \Ri_{\beta}^{\mathbb R^{N}} (N_{\tau}(M))$ or
$\Ri_{\beta}^{d_\tau}(N_{\tau}(M)).$ 
Let $\pi_{(\beta,\tau)}\colon \sh(\Ri_{\beta}(N_{\tau}(M)) \to M$ be the restriction of the projection $\pi\colon N_{\tau}(M) \to M$ to $\sh(\Ri_\beta(N_{\tau}(M)))$.
Then 
the inverse limit homomorphism
\[
\varprojlim_{(\beta,\tau)}\pi_{(\beta,\tau)}\colon \pi_{m}(\sh{\Ri_{\beta}(N_{\tau}(M))}) \to \pi_{m}(M)
\]
is an isomorphism.
\end{proposition}
\begin{proof}
Our proof is a straightforward modification of that of Proposition~\ref{prop:SbetaM}. 
We assume that
$\Ri_\beta(N_{\tau}(M)) = \Ri_{\beta}^{\mathbb R^N}(N_\tau(M))$ in the sequel.
For $\beta, \tau>0$ as above, let $\jmath_{(\beta,\tau)}\colon M \to N_{\tau}(M) \to \sh({\Ri_{\beta}(N_{\tau}(M))})$ be the inclusion. 
%Also for $\gamma<\beta$, $i_{\beta,\gamma}:\sh{\Ri_{\gamma}(M)}) \to \sh{\Ri_{\beta}(M)})$ denotes the inclusion.
We prove the following statement:
for each $(\beta,\tau)>0$, %there exists $0< \gamma<\beta$ such that
\begin{equation}\label{systemisoNbd}
\pi_{(\beta,\tau)}\circ \jmath_{(\beta,\tau)}= \operatorname{id}_{M},~~~~
\jmath_{(\beta,\tau)}\circ \pi_{(\beta,\tau)} \simeq \iota_{(\beta,\tau),(\beta+\varepsilon_{\beta+\tau}, \tau)}.
\end{equation}
The above implies the conclusion as in Proposition~\ref{prop:SbetaM}. 
%Once the above is verifed, we pass to homotopy groups to obtain the equalities
%\[
%\begin{array}{ll}
%\pi_{\gamma}\circ i_{\gamma} = \operatorname{id}_{\pi_{m}(M)}, \\
%i_{\beta} \circ \pi_{\gamma} 
%\simeq i_{\beta,\gamma}:\pi_{m}(\sh(\Ri_{\gamma}(M))) \to
%\pi_{m}(\sh(\Ri_{\beta}(M))),
%\end{array}
%\]
%from which the conclusion follows by a standard argument on inverse limit groups. 
%We choose $\gamma = \beta/2\xi$. 

The first equality is straightforward. For the second homotopy, we take a point $x\in \sh(\Ri_{\beta}(N_{\tau}(M)))$ and find finitely many points $q_{1},\ldots,q_{k}$ of $N_{\tau}(M)$ such that
\[
x \in \conv(\{q_{1},\ldots,q_{k}\})~~\mbox{and}~
\diam_{\mathbb{R}^N}(\{q_{1},\ldots,q_{k}\}) < \beta.
\]
We choose points $p_{1},\ldots, p_{k}$ of $M$ such that $\Vert q_{i}-p_{i}\Vert \leq \tau$ for $i=1, \ldots, k$.
%We see
%\[
%\begin{array}{ll}
%\Vert p_{i}-p_{j}\Vert \leq d_{M}(p_{i},p_{j}) < \gamma,\\
%\diam_{\mathbb{R}^N}\conv(\{p_{1},\ldots,p_{k}\}) = \max_{i,j} \Vert p_{i}-p_{j}\Vert <\gamma.
%\end{array}
%\]
Since $x\in N_{\beta+\tau}(M)$, we have from~\ref{assumptions}~(M3) that
\[
\Vert \pi_{(\beta,\tau)}(x)-x \Vert < \varepsilon_{\beta+\tau}
\]
and hence
\[
\Vert \pi_{(\beta,\tau)}(x)-q_{i} \Vert \leq \Vert \pi_{(\beta,\tau)}(x) - x\Vert + 
\Vert x- q_{i} \Vert < \varepsilon_{\beta+\tau} + \beta
\]
for each $i=1,\ldots,k$.  
Thus
\[
\diam_{\mathbb{R}^N}(\{ \pi_{(\beta,\tau)}(x), q_{1},\ldots,q_{k}\}) < \varepsilon_{\beta+\tau}+  \beta.
\]
The linear homotopy 
$H\colon \sh(\Ri_{\beta}(N_\tau(M)))\times [0,1] \to \sh(\Ri_{\beta + \varepsilon_{\tau}}(N_{\tau}(M))$ defined by
\[
H(x,t) = t \cdot \iota_{(\beta,\tau), (\beta + \varepsilon_{\beta+\tau},\tau)}(x) + (1-t)\cdot \jmath_{(\beta,\tau)}(\pi_{(\beta,\tau)}(x)),~~(x,t) \in 
\sh(\Ri_{\beta}(M))\times [0,1] 
\]
yields the desired conclusion.
\end{proof}

\subsection{Latchev-Type Limit Theorem for Shadow under Noise}
In order to obtain an analog of Theorem~\ref{thm:limitpi}, let $\mathbb{S}_{\tau}$ be
the set of all finite subsets of $N_{\tau}(M)$.  For $S_{1},S_{2} \in \mathbb{S}_\tau$ with $S_{1}\subset S_{2}$, the inclusion
\[
\iota_{\beta}^{S_{1},S_{2}}\colon \sh(\Ri_{\beta}(S_{1})) \to \sh(\Ri_{\beta}(S_{2}))
\]
induces a homorphism of the $m$-homotopy groups 
\[
\iota_{\beta}^{S_{1},S_{2}}\colon \pi_{m}(\sh(\Ri_{\beta}(S_{1}))) \to \pi_{m}(\sh(\Ri_{\beta}(S_{2})))
\]
and hence the direct limit
\[
\pi_{m}(\sh(\Ri_{\beta}(\mathbb{S}_{\tau})))\coloneq  \varinjlim_{S\in \mathbb{S}_\tau}
\{\pi_{m}(\sh(\Ri_{\beta}(S))),  \iota_{\beta}^{S_{1},S_{2}} \mid S_{1},S_{2} \in \mathbb{S}_{\tau}, S_{1}\subset S_{2} \}
\]
is defined.  For each $\beta>0$ and $S\in \mathbb{S}_\tau$, let $\pi_{\beta}^{S}\colon \sh(\Ri_\beta(S)) \to M$ be the restriction of the projection $\pi\colon N(M) \to M$.  Since
$\pi_{\beta}^{S_1} = \pi_{\beta}^{S_2}\circ \iota_{\beta}^{S_{1},S_{2}}$ for each pair $S_{1},S_{2}$ of $\mathbb{S}_\tau$ with $S_{1}\subset S_{2}$, $\{\pi_{\beta}^{S}\mid S\in \mathbb{S}_\tau \}$ induces a homormophism
\[
\pi_{\beta}^{\mathbb{S}_{\tau}}\colon \pi_{m}(\sh(\Ri_\beta(\mathbb{S}_{\tau}))) \to \pi_{m}(M).
\]

%For noisy samples taken from a tubular neighborhood of $M$, the following theorem holds. 

\begin{theorem}\label{limitpi2}
Let $\Ri_{\beta} = \Ri_{\beta}^{\mathbb R^N}$ or $\Ri_{\beta}^{d_\tau}$.
The inverse limit homomorphism
\[
\varprojlim_{(\beta,\tau)}\pi_{\beta}^{\mathbb{S}_{\tau}}\colon  \varprojlim_{(\beta,\tau)}\pi_{m}(\sh(\Ri_\beta(\mathbb{S}_{\tau}))) \to \pi_{m}(M).
\]
is an isomorphism for each $m$.
\end{theorem}

Once again, our proof is a modification of that of Theorem~\ref{thm:limitpi}. For simplicity we give a proof assuming that $\Ri_{\beta} = \Ri_{\beta}^{\mathbb R^{N}}$. The other case follows from this and Proposition \ref{EuclideanM}.

First we introduce a map $f_{(\beta,\tau)}^{S}\colon M\to \Ri_{\beta}(S)$ for $\beta>0$ and $S\in \mathbb{S}_\tau$ which is $\beta/2$-dense in $N_{\tau}(M)$. For a point $q\in N_{\tau}(M)$ and $\varepsilon>0$, let 
\[
B^{\tau}_{\varepsilon}(q) = \{ r\in N_{\tau}(M)\mid \Vert r - q \Vert < \varepsilon\}.
\]
For a point $q\in S$, let
$D^{\tau}_{S}(q) = \{r\in N_{\tau}(M) \mid d_{N_\tau}(r,q) = \min_{x\in S}d_{N_\tau}(x,r)\}$
and choose a continuous function $\lambda_{q}^{S}\colon N_{\tau}(M)\to [0,1]$ such that
\begin{equation}\label{lambda2}
\lambda_{q}^{S}(q) = 1,~~~ \lambda_{q}^{S}|N_{\tau}(M)\setminus D_{S}(q) \equiv 0.
\end{equation}
We define 
\[
\Lambda^{\tau}_{S}(q) = \sum_{x\in B^{\tau}_{\beta/2}(q)\cap S}\lambda_{x}^{S}(q), 
\]
and observe $\Lambda^{\tau}_{S}(q)>0$ for each $q \in N_{\tau}(M)$.

For $\beta>0$, we define a map $f_{(\beta,\tau)}^{S}\colon M\to \sh(\Ri_{\beta}(S))$ as follows.
\begin{equation}\label{fbS-1}
f_{(\beta,\tau)}^{S}(p) = \frac{1}{\Lambda^{\tau}_{S}(p)} \sum_{x\in B^{\tau}_{\beta/2}(p)\cap S} \lambda_{p}^{S}(x)\cdot x
\end{equation}
We have from the definition:
\[
f_{(\beta,\tau)}^{S}(p) \in \conv(B_{\beta/2}^{\tau}(p)\cap S).
\]
%and hence $f_{(\beta,\tau)}^{S}$ is indeed a map to $\sh(\Ri_{\beta}(S))$.

The proof of the following lemma is the same as that of Lemma \ref{inclusions}.

\begin{lemma}\label{inclusions2}
For $S, S_{1}, S_{2} \in \mathbb{S}_\tau$ and $0< \beta, \beta_{1} < \beta_{2}$ with
$S_{1}\subset S_{2}$ and $\beta_{2} < \beta_{1}$, 
let $\iota_{\beta}^{S_{1},S_{2}}\colon \sh(\Ri_{\beta}(S_{1})) \to \sh(\Ri_{\beta}(S_{2}))$ and
$\iota_{\beta_{1},\beta_{2}}^{S}\colon \sh(\Ri_{\beta_2}(S)) \to \sh(\Ri_{\beta_1}(S))$ be the inclusions. 
We have the following.
\begin{itemize}
\item[(1)] $\iota_{\beta}^{S_{1},S_{2}} \circ f_{(\beta, \tau)}^{S_{1}} \simeq f_{(\beta,\tau)}^{S_{2}}.$
\item[(2)] $\iota_{\beta_{1},\beta_{2}}^{S} \circ f_{(\beta_{2},\tau)}^{S} \simeq f_{(\beta_{1},\tau)}^{S}.$
\end{itemize}
\end{lemma}

%\begin{proof}
%(1).  For $\beta>0$ and $S_{1}, S_{2} \in \mathbb S$ with $S_{1} \subset S_{2}$, take a point $p$ of $M$.  By the definition of $f_{\beta}^{S_1}$ and $f_{\beta}^{S_2}$ we have
%\[
% f_{\beta}^{S_1}(p) \in \conv(S_{1}\cap B_{\beta/2}(p)),~\mbox{and}~~
% f_{\beta}^{S_2}(p) \in \conv(S_{2}\cap B_{\beta/2}(p)).
%\]
%By the assumption, $\conv(S_{2}\cap B_{\beta/2}(p)) \supset
%\conv(S_{1}\cap B_{\beta/2}(p))$, thus we see for each $t\in [0,1]$,
%\[
%(1-t) ~i_{\beta}^{S_{1},S_{2}} (f_{\beta}^{S_{1}}(p)) + t~ f_{\beta}^{S_{2}}(p) 
%\in \conv(S_{2}\cap B_{\beta/2}(p)) \subset
%\sh(\Ri_\beta(S_{2})).\]
%Thus the map $M\times [0,1] \to \sh(\Ri_\beta(S_{2}))$ defined by 
%\[
%t\mapsto (1-t) i_{\beta}^{S_{1},S_{2}} \circ f_{\beta}^{S_{1}} + t f_{\beta}^{S_{2}} 
%\]
%gives the desired homotopy between $i_{\beta}^{S_{1},S_{2}} \circ f_{\beta}^{S_{1}}$ and $f_{\beta}^{S_{2}}$.
%This proves (1).  The proof of (2) is similar to the above.
%
%\end{proof}

For $\beta>0$ and $S\in \mathbb{S}_\tau$, $\pi_{\beta}^{S}\colon \sh(\Ri_{\beta}(S)) \to M$ be the restriction of the projection $\pi\colon N(M)\to M$ to the space $\sh(\Ri_{\beta}(S))$.

\begin{lemma}\label{pifbs2}
%For $\beta, \tau>0$ satisfying ($\beta\tau$-1),($\beta\tau$-2)  and ($\beta\tau$-3) and $S\in \mathbb{S}_{\tau}$, we have the following.
Assume that $\beta$ and $\tau$ satisfy:
\[
 \beta+\varepsilon_{\beta} + \varepsilon_{\beta+\tau}< \delta,~~~\xi(\beta+\varepsilon_{\beta}) < \rho(M).
\]
%and also assume that $B^{\tau}_{\beta/2}(p)$ is defined by the Euclidean metric. 
For each $\beta/2$-dense finite subset $S$ of $N_\tau(M)$, we have the following.
 
\begin{itemize}
\item[(1)] $\pi_{\beta}^{S} \circ f_{(\beta,\tau)}^{S} \simeq \operatorname{id}_{M}$.
\item[(2)] Let $\mu_{(\beta,\tau)}= 
 (3\beta/2) + 2\varepsilon_{\tau}+ \varepsilon_{\beta+\tau}. $
%(2\xi + \frac{1}{2})\beta+(2\xi+1)\tau$. 
Then we have
$\displaystyle
\iota_{\mu_{(\beta,\tau)},\beta}^{S}\circ f_{(\beta,\tau)}^{S} \circ \pi_{\beta}^{S} \simeq \iota_{\mu_{(\beta,\tau)},\beta}^{S}\colon \sh(\Ri_{\beta}(S)) \to \sh(\Ri_{\mu_{(\beta,\tau)}}(S)).
$
\end{itemize}
\end{lemma}
\begin{proof} 

Take $\beta, \tau$ and $S\in \mathbb{S}_\tau$ as in the hypothesis.

(1)  For a point $p$ of $M$, we have $f_{(\beta,\tau)}^{S}(p) \in \conv(S \cap B^{\tau}_{\beta/2}(p))$ and 
$\displaystyle \diam_{\mathbb{R}^{N}}(S \cap B^{\tau}_{\beta/2}(p))  < \beta.$

For a point $x \in S\cap B_{\beta/2}^{\tau}(p)$, we observe
\begin{eqnarray*}
\Vert f_{\beta,\tau}^{S}(p) - p \Vert &\leq& \Vert f_{\beta,\tau}^{S}(p)-x\Vert + 
\Vert x - p\Vert \\
&\leq& (\beta/2) + (\beta/2) = \beta.
\end{eqnarray*}
Hence $f_{\beta,\tau}^{S}(p) \in N_{\beta}(M)$ and we see
\[
\Vert \pi_{\beta}^{S}(f_{(\beta,\tau)}^{S}(p))-f_{(\beta,\tau)}^{S}(p)\Vert \leq
\varepsilon_{\beta}
\]
by~\ref{assumptions}(M3).
These two imply
\[
\Vert p-\pi_{\beta}^{S}(f_{(\beta,\tau)}^{S}(p))\Vert \leq \Vert p-f_{(\beta,\tau)}^{S}(p) \Vert + \Vert f_{(\beta\tau)}^{S}(p)-\pi_{\beta}^{S}(f_{(\beta,\tau)}^{S}(p))\Vert 
\le \beta+\varepsilon_{\beta} < \delta.
\]
Thus we obtain $d_{M}(p,\pi_{\beta}^{S}(f_{\beta}^{S}(p))) < \xi(\beta+\varepsilon_\beta) < \rho(M)$ for each $p\in M$.  Hence $\pi_{\beta}^{S}\circ f_{\beta}^{S}$ is $\rho(M)$-close to $\operatorname{id}_M$.  Hence, these maps are homotopic.

%$\gamma_{p}:[0,1]\to M$ with $\gamma_{p}(0) = \pi_{\beta}(f_{\beta}^{S}(p))$ and $\gamma_{p}(1) = p$.
%Define a homotopy $H:M\times [0,1]\to M$ by $H(p,t) = \gamma_{p}(t)$. 
%Since $p \mapsto \gamma_p$ varies continuously, we see that $H$ is a homotopy between $\pi_{\beta}\circ f_{\beta}^{S}$ and $\operatorname{id}_{M}$.
(2).  For a point $x\in \sh(\Ri_{\beta}(S))$, there exist points $q_{1},\ldots,q_{k}$ of $N_{\tau}(M)$ such that
\[
x\in \conv(\{q_{1},\ldots,q_{k}\}),~~~ \diam_{\mathbb{R}^N}(\{q_{1},\ldots,q_{k}\}) < \beta.
\]
We observe that $x \in N_{\beta+\tau}(M)$.
For each $i=1,\ldots,k$, let $p_{i} = \pi_{\beta}^{S}(q_{i}) \in M$. Since $q_i \in N_{\tau}(M)$, we have from~\ref{assumptions}~(M3) that 
\[
\Vert p_{i}-q_{i}\Vert \leq \varepsilon_\tau.
\]
Then
\begin{eqnarray*}
\Vert \pi_{\beta}^{S}(x) -p_{i}\Vert &\leq& \Vert \pi_{\beta}^{S}(x)-x\Vert + \Vert x-q_{i}\Vert + \Vert q_{i} - p_{i}\Vert \\
&\leq&  \varepsilon_{\beta+\tau}+\beta+\varepsilon_\tau.
%\diam_{\mathbb{R}^N}(\{p_{1},\ldots,p_{k}\}) \leq 2 \diam_{M}(\{p_{1},\ldots, p_{k}\}) \\
%&<& 2\beta.
\end{eqnarray*}
The last term is less than $\delta$ and hence by~\ref{assumptions}~(M2), we obtain
\[
d_{M}(\pi_{\beta}^{S}(x),p_{i}) < \xi (\beta+\varepsilon_{\beta+\tau}+\varepsilon_{\tau}),~~~i=1,\ldots,k.
\]
Now, for a point $y\in S\cap B^{\tau}_{\beta/2}(\pi_{\beta}^{S}(x))$, we see
\begin{eqnarray*}
\Vert y - q_{i} \Vert &\leq& \Vert y- \pi_{\beta}^{S}(x) \Vert 
+ \Vert \pi_{\beta}^{S}(x) - p_{i} \Vert + \Vert p_{i} -q_{i} \Vert \\
&<& (\beta/2) + (\beta+\varepsilon_{\tau}+\varepsilon_{\beta+\tau}) + \varepsilon_{\tau} = (3\beta/2) + 2\varepsilon_{\tau}+ \varepsilon_{\beta+\tau} \\
&=& \mu_{(\beta,\tau)}.
\end{eqnarray*}
It follows from the above that
\[
\diam_{\mathbb{R}^N}(\{q_{1},\ldots,q_{k}\}\cup (S\cap B^{\tau}_{\beta/2}(\pi_{(\beta,\tau)}(x))) < \mu_{(\beta,\tau)},
\]
and the conclusion (2) follows as in Lemma \ref{pifbs} (2).
%
%Hence for each $t\in [0,1]$, 
%\[
%(1-t)x+ t f_{\beta}^{S}(\pi_{\beta}(x)) \in \conv(\{p_{1},\ldots,p_{k}\}\cup (S\cap B_{\beta/2}(\pi_{\beta}(x)))) \subset \sh(\Ri_{\nu_{\beta}}(S)).
%\]
%The map $\sh(\Ri_{\beta}(S))\times [0,1] \to \sh(\Ri_{\nu_\beta}(S))$ defined by $(x,t) \mapsto (1-t)x+ t f_{\beta}^{S}(\pi_{\beta}(x))$ gives a homotopy between the maps $i_{\nu_{\beta}}^{S}$ and $i_{\nu_\beta}^{S}\circ f_{\beta}^{S} \circ \pi_{\beta}$.  This proves (2).
\end{proof}

Having these lemmas, Theorem \ref{limitpi2} is proved in exactly the same way as that of Theorem~\ref{thm:limitpi}. 
\begin{flushright}
$\Box$
\end{flushright}

\begin{remark}
We can make use of Lemma~\ref{pifbs2} to obtain information on the homotopy group of $M$ as follows: 
fix a pair $(\beta,\tau)$ satisfying the hypothesis of Lemma~\ref{pifbs2} and a $\beta/2$-dense finite subset $S$ of $N_\tau(M)$,
\begin{itemize}
    \item [(i)] Since the map $f_{(\beta,\tau)}^{S}$ induces a monomorphism on homotopy groups by (1) of Lemma~\ref{pifbs2}, the homotopy group $\pi_{m}(M)$ is isomorphic to a subgroup of $\pi_{m}(\sh(\Ri_\beta(S)))$.  In particular, 
    $\pi_{m}(\sh(\Ri_\beta(S))) = 1$ implies that 
    $\pi_{m}(M) = 1$. If $\pi_{m}(\sh(\Ri_\beta(S)))$ is abelian, then so is $\pi_{m}(M)$ and $\operatorname{rank}(\pi_{m}(M)) \leq \operatorname{rank}(\pi_{m}(\sh(\Ri_\beta(S)))$.
   %Then, since $f_{(\beta,\tau)}^{S}$ induces a monomorphism on homotopy groups by (1) of the lemma, we see that $\pi_{m}(M) = 0$. 
    \item[(ii)] By Lemma~\ref{pifbs2} (2), we have
    \[
    \operatorname{Ker}(\pi_{\beta}^{S}) \subset \operatorname{Ker}(\iota_{\mu(\beta,\tau),\beta}^{S})
    \]
    Hence if we find a non-trivial element $\omega$ of $\pi_{m}(\sh(\Ri_{\beta}(S)))$ that survives in $\pi_{m}(\sh(\Ri_{\mu_{(\beta,\tau)}}(S)))$, then $\pi_{\beta}^{S}(\omega)$ is a non-trivial element of $\pi_{m}(M)$. 
    \end{itemize}
Since we have ``quantitative'' estimates for $\beta,\tau$ and $S$ as indicated in Lemma~\ref{pifbs2}, the above observation may be regarded as a partial quantitative estimate on the homotopy group of $M$.
\end{remark}

\subsection{Limit Theorem for Shadow Projection under Noise}

The following assumption is motivated by the Hausmann-type theorems for metric graphs~\cite{majhi2023vietoris} and metric spaces with bounded curvature~\cite{MajhiStability}. For these spaces, the metric $d_\tau$ was chosen as the $\epsilon$-path metric $d^\varepsilon$ for a sufficiently small $\varepsilon$.  For a smooth submanifold, $d_{\tau}$ is chosen as the Euclidean distance.

\bigskip
\noindent
{\bf Assumption (G)}~ For any sufficiently small $\beta, \tau>0$, there associate 
$\eta_{\beta,\tau}>0$ with $\displaystyle \lim_{(\beta,\tau) \to (0,0)}\eta_{\beta,\tau} = 0$ such that the simplicial map $\psi_{\beta,\tau}\colon \Ri^{d_\tau}(N_{\tau}(M)) \to \Ri_{\eta_{\beta,\tau}}(M)$ induced by the projection $\pi_{(\beta,\tau)}\colon N_{\tau}(M) \to M$ is a homotopy equivalence.

\bigskip
Under the above assumption, we study the Vietoris-Rips projection $p_{\beta}^{S}\colon \Ri_{\beta}(S) \to \sh(\Ri_\beta(S))$ of a sample set $S$ in a neighborhood of $M$. Our theorem is stated as follows:

\begin{theorem}\label{limitprojection2}
%\begin{itemize}
%\item[(1)] The inclusion 
%\[
%M \to N_{\tau}(M) \to \sh(\Ri_\beta(N_{\tau}(M)))
%\]
%induces an isomorphism
%\[
%\ell_{\infty}\coloneq  \varinjlim_{(\beta,\tau)} \ell_{(\beta,\tau)}^{\mathbb{S}_\tau}:\pi_{m}(M) %\to \varprojlim_{(\beta,\tau)}\pi_{m}(\Ri_{\beta}(\mathbb{S}_{\tau}))
%\]
%\item[(2)]
%The composition of $\ell_\infty$ and the inverser limit homorphism induced by the system of direct limit homomoprhisms
%$\{ p_{\beta}^{\mathbb{S}_{\tau}}:\pi_{m}(\Ri_\beta(\mathbb{S}_{\tau})) \to \pi_{m}(\sh(\Ri_\beta(\mathbb{S}_{\tau}))) \mid \beta,\tau > 0\}$:
%\[
%\left(\varprojlim_{(\beta,\tau)}p_{(\beta,\tau)}^{\mathbb{S}_{\tau}}\right) \circ \ell_{\infty}:\pi_{m}(M) \to \varprojlim_{(\beta,\tau)} \pi_{m}(\Ri_{\beta}(\mathbb{S}_{\tau})) \to \varprojlim_{(\beta,\tau)}\pi_{m}(\sh(\Ri_\beta(\mathbb{S}_{\tau})))
%\]
%is an isomorphism.
Under the assumptions (G), (H) and (N), 
the inverse limit homomorphism induced by the system of direct limit homomorphisms
$\{ p_{\beta}^{\mathbb{S}_{\tau}}\colon \pi_{m}(\Ri_\beta^{d_\tau}(\mathbb{S}_{\tau})) \to \pi_{m}(\sh(\Ri_\beta^{d_\tau}(\mathbb{S}_{\tau}))) \mid \beta,\tau > 0\}$:
\[
\left(\varprojlim_{(\beta,\tau)}p_{(\beta,\tau)}^{\mathbb{S}_{\tau}}\right)\colon \varprojlim_{(\beta,\tau)} \pi_{m}(\Ri_{\beta}^{d_\tau}(\mathbb{S}_{\tau})) \to \varprojlim_{(\beta,\tau)}\pi_{m}(\sh(\Ri_\beta^{d_\tau}(\mathbb{S}_{\tau})))
\]
is an isomorphism.
%%\end{itemize}
\end{theorem}
%\begin{proof}
%(1)
%The direct limit map 
%\[
%\varinjlim_{S\in \mathbb{S}_{\tau}}i_{\beta}^{S}:\pi_{m}(\Ri_\beta(N_{\tau}(M))) \to %\pi_{m}(\Ri_{\beta}(\mathbb{S}_{\tau}))
%\]
%is an isomorophism because $\Ri_\beta(N_{\tau}(M))$ is endowed with Whitehead topology. Composing the above with the inclusion $i_{(\beta,\tau)}:M\to N_{\tau}(M)$ which is a homotopy equivalence, we have an isomorphism
%\[
%\ell_{\beta}^{S}\coloneq  i_{\beta}^{S} \circ i_{(\beta,\tau)}:\pi_{m}(M) \to \pi_{m}(\Ri_{\beta}(\mathbb{S}_{\tau}))
%\]
%which induces an isomorphism on the inverse limits:
%
%\[
%\varinjlim_{(\beta,\tau)} \ell_{(\beta,\tau)}^{\mathbb{S}_\tau}:\pi_{m}(M) \to \varprojlim_{(\beta,\tau)}\pi_{m}(\Ri_{\beta}(\mathbb{S}_{\tau}))
%\]
%
%\end{proof}

%\vspace{5mm}

Before we give a proof of the theorem at the end of the current section, we first need a couple of preparations.

For sufficiently small $\delta_{0} > \beta, \tau>0$, let $\eta_{\beta,\tau}$ be the positive number in Assumption (G). For a finite set $S\in \mathbb{S}_\tau$, we consider the following diagram:
\begin{equation}\label{Rprojection}
\xymatrix{
\Ri^{\mathbb{R}^N}_{\beta}(S) \ar[r]^{p_{\beta}^{S}} 
\ar[d]_{k_{\beta,\tau}^{S}} & 
\sh(\Ri^{\mathbb {R}^N}_{\beta}(S)) \ar[d]^{\sh(k_{\beta,\tau}^{S})} \\
\Ri^{d_\tau}_{\kappa_{1}\beta}(S) \ar[r]^{p_{\beta}^{S}} 
\ar[d]_{\psi_{\kappa_{1}\beta}^{S}} & 
\sh(\Ri^{d_\tau}_{\kappa_{1}\beta}(S)) \ar[d]^{\pi_{\beta}^{S}} \\
\Ri_{\eta_{\kappa_{1}\beta,\tau}}(M) \ar[r]_{T_{\eta_{\kappa_{1}\beta}}} & M 
}
\end{equation}
where $\psi_{\kappa_{1}\beta}^{S}$ is the simplicial map given as in Assumption (G) 
and $T_{\eta_{\kappa_{1}\beta}}$ denotes the Hausmann map given in (H).  Also $k_{\beta,\tau}^{S}$ and $\sh(k_{\beta,\tau}^{S})$ denote the inclusions (see Assumption (N)).
%It follows from the definition that the right triangle of the above is commutaitve:~$\pi_{\beta}^{S} = \pi_{\beta}^{S}\circ j_{\beta}^{S}$.

\begin{lemma}\label{psipi}
Assume that $S \in \mathbb{S}_{\tau}$ or $S=N_{\tau}(M)$ and assume that
$\beta, \tau$ satisfy 
\[
\beta < \delta_{0},~~2(\beta+ \varepsilon_{\tau}) < \delta, ~~ 2\xi(\beta+\varepsilon_{\tau})+\eta_{\beta,\tau} <\rho(M).
\] 
Then we have
\[
T_{\beta}\circ \psi_{\beta}^{S} \circ k_{\beta}^{S} \simeq \pi_{\beta}^{S}\circ p_{\beta}^{S} \circ \sh(k_{\beta,\tau}^{S}).
\]
\end{lemma}
\begin{proof}

Take an arbitrary simplex $\sigma = [x_{0},\ldots, x_{k}]$ of 
$\Ri_\beta^{\mathbb{R}^N}(S)$, where $\{ x_{0},\ldots, x_{k} \} \subset S$ and $\diam_{\mathbb{R}^N}\{ x_{0},\ldots,x_{k} \} < \beta$.  By Assumption (N) we have
$\diam_{d_\tau}(\{ x_{0},\ldots,x_{k} \}) < \kappa_{1} \beta$ and
$\diam_{d_M}(\{ \pi_{\beta}^{S}(x_{0}),\ldots, \pi_{\beta}^{S}(x_{k}) \}) < \kappa_{1} \beta$
.  
Applying Assumption (H), we see that 
\[
d_{M}(T_{\eta_{\kappa_{1}\beta}}(\psi_{\kappa_{1}\beta}^{S}(x)), \pi_{\beta}^{S}(x_{i})) < \eta_{\kappa_{1}\beta}.
\]
On the other hand, we see $p_{\beta}^{S}(x) \in \conv(\{x_{0},\ldots,x_{k}\})$ and 
$\diam_{\mathbb{R}^N}(\conv(\{x_{0},\ldots,x_{k}\})) < \beta$. 
We have
\begin{eqnarray*}
\Vert \pi_{\beta}^{S}(p_{\beta}^{S}(x)) - \pi_{\beta}^{S}(x_{j})  \Vert 
&\leq& \Vert \pi_{\beta}^{S}(x_{j}) - p_{\beta}^{S}(x)\Vert + \Vert p_{\beta}^{S}(x) - x_{i} \Vert + \Vert x_{i}-\pi_{\beta}^{S}(x_{i})\Vert \\
&\leq& (\beta+\varepsilon_{\tau}) + \beta + \varepsilon_{\tau} < \delta.
\end{eqnarray*}
Hence we obtain
\begin{equation}\label{tpsi2}
d_{M}(\pi_{\beta}^{S}(p_{\beta}^{S}(x)) , \pi_{\beta}^{S}(x_{j})) < 2\xi(\beta+\varepsilon_{\tau}).
\end{equation}
These two imply
\[
d_{M}(\pi_{\beta}^{S}(p_{\beta}^{S}(x)) , T_{\eta_{\kappa_{1}\beta}}(\psi_{\kappa_{1}\beta}^{S}(x))) < \eta_{\kappa_{1}\eta}+ 2\xi(\beta+\varepsilon_{\tau}) <\rho(M).
\]
From this, we obtain the desired conclusion.
%Using (\ref{tpsi1}) and (\ref{tpsi2}), we have
%\begin{eqnarray*}
%d_{M}(T_{\beta}(\psi_{\beta}^{S}(x)), \pi_{\beta}^{S}(p_{\beta}^{S}(x))) 
%&\leq&  d_{M}(T_{\beta}(\psi_{\beta}^{S}(x)), \pi_{\beta}^{S}(x_{j}) + d_{M}(\pi_{\beta}^{S}(x_{j}), \pi_{\beta}^{S}(p_{\beta}^{S}(x))) \\
%&<&\delta_{\beta}+2\xi(1+\zeta)\beta <\rho(M)
%\end{eqnarray*}
%by the choice of $\beta$. Hence $T_{\beta}(\psi_{\beta}^{S}(x))$ and $\pi_{\beta}^{S}(p_{\beta}^{S}(x)))$ are joined by a unique geodesic and we obtain the desired conclusion.
% 

\end{proof}

We finally prove Theorem~\ref{limitprojection2}.
\begin{proof}[Proof of Theorem~\ref{limitprojection2}]
    
Passing to the homotopy groups in the diagram (\ref{Rprojection}) using 
Assumption (G), Assumption (H) and Lemma \ref{psipi} and furthermore, taking the direct limits, we obtain the next  commutative diagram:
\begin{equation}
\xymatrix{
\pi_{m}(\Ri^{\mathbb{R}^N}_{\beta}(\mathbb{S}_{\tau})) \ar[r]^{p_{\beta}^{\mathbb{S}_{\tau}}} \
\ar[d]_{k_{\beta,\tau}^{\mathbb{S}_{\tau}}} & 
\pi_{m}(\sh(\Ri^{\mathbb {R}^N}_{\beta}(\mathbb{S}_{\tau}))) \ar[d]^{\sh(k_{\beta,\tau}^{\mathbb{S}_{\tau}})} \\
\pi_{m}(\Ri^{d_\tau}_{\kappa_{1}\beta}(\mathbb{S}_{\tau})) \ar[r]^{p_{\beta}^{\mathbb{S}_{\tau}}} 
\ar[d]_{\psi_{\kappa_{1}\beta}^{\mathbb{S}_{\tau}}} & 
\pi_{m}(\sh(\Ri^{d_\tau}_{\kappa_{1}\beta}(\mathbb{S}_{\tau}))) \ar[d]^{\pi_{\beta}^{\mathbb{S}_{\tau}}} \\
\pi_{m}(\Ri_{\eta_{\kappa_{1}\beta,\tau}}(M)) \ar[r]_{T_{\eta_{\kappa_{1}\beta}}} & \pi_{m}(M) 
}
\end{equation}

%\xymatrix{
%\pi_{m}(\Ri^{\mathbb{R}^{N}}_{\beta}(\mathbb{S}_{\tau})) 
%\ar[r]^{p_{\beta}^{\mathbb{S}_{\tau}}} 
%\ar[d]_{\psi_{\beta}^{\mathbb{S}_{\tau}}} 
%& 
%\pi_{m}(\sh(\Ri^{\mathbb {R}^{N}}_{\beta}(\mathbb{S}_{\tau}))) 
%\ar[d]^{\pi_{\beta}^{\mathbb{S}_\tau}} 
%& 
%\pi_{m}(\sh(\Ri_{\beta}(\mathbb{S}_{\tau})) 
%\ar[l]_{j_{\beta}^{\mathbb{S}_{\tau}}} 
%\ar[ld]^{\pi_{\beta}^{\mathbb{S}_\tau}} 
%\\
%\pi_{m}(\Ri_{\delta_{\beta}}(M)) \ar[r]_{T_\beta} & \pi_{m}(M) 
%}

Here $\psi_{\kappa_{1}\beta}^{\mathbb{S}_\tau}$ is an isomorphism by Assumption (G) and $T_{\eta_{\kappa_{1}\beta}}$ is an isomorphism by Assumption (H). 

Now, we pass to the inverse limits $\varprojlim_{(\beta,\tau)}$.  The limit homomorphism $\varprojlim_{(\beta,\tau)}k_{\beta}^{\mathbb{S}_{\tau}}$ is an isomorphism by Lemma \ref{EuclideanM}. Moreover, the limit homomorphism
\[
\varprojlim_{(\beta,\tau)} (\pi_{\beta}^{\mathbb{S}_{\tau}} \circ \sh(k_{\beta,\tau}^{\mathbb{S}_{\tau}}))\colon 
\varprojlim_{(\beta,\tau)} \pi_{m}(\sh(\Ri_{\beta}^{\mathbb{R}^N}(\mathbb{S}_{\tau}))) \to  \pi_{m}(M)
\]
is an isomorphism by Theorem \ref{limitpi2}.  Hence, we see that 
$\displaystyle \varprojlim_{(\beta,\tau)}p_{\beta}^{\mathbb{S}_\tau}$ is an isomorphism.  

This completes the proof.
\end{proof}

\section{Towards Finite Reconstruction of Closed Curves}\label{sec:recon}
In view of our Question (b) in Section~\ref{problemSetup}, we may ask under which condition the above limit process is actually stabilized. 
In this section, we address the question of finite reconstruction: under what conditions is the shadow $\sh(\Ri_\beta(S))$ homotopy equivalent to $M$ for a (possibly finite) subset $S\subset\R^N$ with sufficiently small $d_\h(M, S)$?

When $M$ is a one-dimensional closed smooth submanifold, that is, a smooth simple closed curve, we obtain some such results in Theorem~\ref{onedim} as follows.  
We not only reconstruct the homotopy of $M$ but also reconstruct its topological embedding type.
%knot type (when $N=3$).

In the sequel, $M$ is a smooth simple closed curve in $\mathbb R^N$ and $N(M)$ in \ref{assumptions}~(M3) is a tubular neighborhood of $M$ with the bundle projection 
$\pi\colon N(M)\to M$. It satisfies 
\[
\Vert \pi(x)-x\Vert = \min_{p\in M}\Vert p-x\Vert
\]
for each point $x \in N(M)$. Also for $\tau >0$, $N_{\tau}(M)$ denotes the $\tau$-tubular neighborhood of $M$.
For a point $p\in M$, $T_{p}(M)$ denotes the tangent line of $M$ at $p$, regarded as an affine line of $\mathbb{R}^N$ through $p$. Also $(T_{p}(M))^{\perp}$ denotes the affine subspace of $\mathbb{R}^N$ through $p$ that is orthogonal to $T_{p}(M)$.
For a point $x\in \mathbb{R}^N$ and $\varepsilon>0$, let $B_{\mathbb{R}^N}(x,\varepsilon) = \{y\in \mathbb{R}^{N}\mid \Vert x-y\Vert < \varepsilon\}$.  
Due to $M$ being one-dimensional, there exists a positive number $\eta(M)$ such that
\[
(T_{p}(M))^{\perp} \cap (B_{\mathbb{R}^N}(p,\eta) \cap M) = \{p\}
\]
for each $p \in M$ and $\eta<\eta(M)$.

We consider the following conditions on the scale $\beta$. 
\begin{itemize}
\item[($\beta$-1)] $\sh(\Ri_\beta (M)) \subset N(M)$.
%, where $N(M)$ is the neighborhood of $M$ in \ref{assumptions}(M1).
\item[($\beta$-2)] $3 \beta < \eta(M)$,
\item[($\beta$-3)] $\beta + \xi(\beta+\varepsilon_{\beta}) < \rho(M)$, where $\xi$ and $\varepsilon_{\beta}$ are constants in~\ref{assumptions}~(M2) and (M3) for the induced metric $d_M$ on $M$ as a Riemannian submanifold of $\mathbb{R}^N$.

\end{itemize}
When $\sh(\Ri_{\beta}(M)) \subset N(M)$, the restriction of $\pi:N(M)\to M$ to $\sh(\Ri_{\beta}(M))$ is denoted by $\pi_\beta$.  As in the previous sections, $\mathbb S$ and $\mathbb{S}_{\tau}$ denote the collections of all finite subsets of $M$ and $N_{\tau}(M)$ respectively.

\begin{theorem}\label{onedim}
Let $M$ be a one-dimensional smooth closed submanifold of $\mathbb{R}^N$. Assume that  $\beta>0$ satisfies the conditions ($\beta$-1), ($\beta$-2) and ($\beta$-3).
\begin{itemize}
\item[(1)] %Assume $\beta$ satisfies ($\beta$-1) and ($\beta$-2).  Then 
The projection $\pi_{\beta}\colon\sh(\Ri_\beta(M)) \to M$  and the shadow projection
$p_{\beta}\colon\Ri_\beta(M) \to \sh(\Ri_\beta(M))$ are homotopy equivalences.
\item[(2)] Let $S\in \mathbb{S}$ such that $S$ is $\beta/2$-dense. 
The projection $\pi_{\beta}^{S}\colon\sh(\Ri_\beta(S)) \to M$ and 
the shadow projection $p_{\beta}^{S}\colon \Ri_\beta(S) \to \sh(\Ri_\beta(S))$ are homotopy equivalences.
\item[(3)]
Assume $\tau>0$ and $S \in \mathbb{S}_\tau$ satisfies $N_{\tau}(M) \subset N(M)$ and $\pi(S)$ is $\zeta$-dense.  If $\tau + \zeta < \beta/2$, then there exists a PL simple closed curve $K \subset \sh(\Ri^{\mathbb{R}^N}_{\beta}(S))$ such that $K$ and $M$ are topologically equivalently embedded in $\mathbb{R}^N$.
\end{itemize}
\end{theorem}

%[[The following comments are formulated in the  remark below.]]

%\sush{Compare with Henry's metric thickening \cite{Adams_2019}, where only $\pi_m$-surjectivity of the shadow projection has been proved for Euclidean Rips $\Ri^{\R^N}_\beta(M)$.}

%\sush{Using result (1), a little bit of diagram chasing can show that shadow projection is $\pi_m$-injective on $\Ri^{\R^N}_\beta(M)$. Can surjectivity also be proved?}

%\sush{Atish and I are trying to prove this for Euclidean Rips $\Ri^{\R^N}_\beta(S)$ in the noisy case.}

\begin{remark}
\begin{itemize}
\item[(1)]
The metric thickening, denoted by $\Ri_{\beta}^{\mathfrak{M}}(M)$ in the present paper,  for a metric space $M$ with scale parameter $\beta$ was introduced in \cite{Adamaszek2018}. There exists a natural continuous bijection $j\colon\Ri_{\beta}(M) \to \Ri^{\mathfrak{M}}_{\beta}(M)$ that induces an isomorphism in homotopy groups in all dimention \cite[Theorem 1]{gillespie2024}.  For a subset $M$ of $\mathbb{R}^N$ with the induced metric, $\Ri^{m}_{\beta}(M)$ admits a natural map $f:\Ri_{\beta}^{\mathfrak{M}}(M) \to \sh(\Ri^{\mathbb{R}^N}(M))$ such that the shadow projection is equal to the composition of $f$ and $j$:
\[
p= f\circ j\colon\Ri_{\beta}^{\mathbb{R}^N}(M) \to \sh(\Ri_{\beta}^{\mathbb{R}^N}(M)).
\]
The proofs of Adamszek-Adams \cite[Theorem 4.6]{Adams_2019} and Gillespie \cite[Theorem 1]{gillespie2024} show that the composition map $\pi_{\beta}\circ p_{\beta}\colon\Ri^{\mathbb{R}^N}_{\beta}(M) \to \sh(\Ri_{\beta}^{\mathbb{R}^N}(M)) \to M$ induces an isomorphism of homotopy groups in all dimensions.  It follows from this that the induced homomorphisms by $\pi_\beta$ and $p_\beta$ are surjective and injective, respectively. 
\item[(2)] The proofs of Proposition \ref{onedim} (1) and (2) are not carried over for the Euclidean Vietoris-Rips complex $\Ri^{\mathbb{R}^N}_{\beta}(M)$.
\item[(3)] For the PL simple curve $K$ of (3) of the above theorem, one can show that there exists a retraction $r:\sh(\Ri_{\beta}^{\mathbb R^N}(S)) \to K$ such that
$r$ is homotopic to the inclusion $\sh(\Ri_{\beta}^{\mathbb R^N}(S)) \to \sh(\Ri_{2\beta}^{\mathbb R^N}(S))$. However it is not known whether $r$ is homotopic to $\operatorname{id}_{\sh(\Ri_{\beta}(S))}$. Notice that $K$ is homeomorphic to $M$.
\end{itemize}
\end{remark}

\bigskip

We start with a lemma.

\begin{lemma}\label{intermediate}
Assume that $\beta$ satisfies ($\beta$-1) , ($\beta$-2) and ($\beta$-3).  
Let $p_{1},\ldots,p_{k}$ be points of $M$ such that $d_{M}(p_{i},p_{j}) < \beta$ for each $i,j=1,\ldots,k$, and let $C_{M}(p_{1},\ldots,p_{k})$ be the minimum curve of $M$ containing the set $\{p_{1},\ldots, p_{k}\}$.  For each point 
$x \in \conv(\{p_{1},\ldots, p_{k}\})$, we have $\pi(x) \in C_{M}(p_{1}, \ldots,p_{k})$ 
\end{lemma}

\begin{proof}[Proof of Lemma~\ref{intermediate}]
Let $x$ be a point as in the hypothesis. 
Let $H^+$ and $H^-$ be the closed half-spaces of $\mathbb{R}^N$ determined by the $(N-1)$-dimensional hyperplane $(T_{\pi(x)}(M))^{\perp}$.
Also let $\Pi:\mathbb{R}^{N}\to T_{\pi(x)}(M)$ be the orthogonal projection onto the tangent line $T_{\pi(x)}(M)$ at $\pi(x)$. Observe that $\Pi(x) = \Pi(\pi(x)) = \pi(x)$, and $x\in (T_{\pi(x)}(M))^{\perp}$.

Since $\pi(x) = \Pi(x) \in \conv_{T_{\pi(x)}(M)}(\Pi(p_{1}),\ldots,\Pi(p_{k}))$, we have
\[
\{p_{1}\ldots,p_{k}\} \cap H^{+} \neq \emptyset \neq \{p_{1},\ldots,p_{k}\}\cap H^{-}.
\]
Hence we obtain
\[
C_{M}(p_{1}\ldots,p_{k}) \cap H^{+} \neq \emptyset \neq  C_{M}(p_{1}\ldots,p_{k})\cap H^{-},
\] 
which implies $C_{M}(p_{1}\ldots,p_{k}) \cap (T_{\pi(x)}(M))^{\perp} \neq \emptyset.$ 
Observe 
\[
\Vert \pi(x) -x \Vert \leq \Vert p_{i}-x \Vert
\]
because $p_i \in M$.
For each point $q \in C_{M}(p_{1},\ldots,p_{k})$, we see
\begin{eqnarray*}
\Vert q-\pi(x)\Vert &\leq& \Vert q-p_{1}\Vert + \Vert p_{1}-x\Vert + \Vert x-\pi(x)\Vert \\
&\leq& \Vert q-p_{1}\Vert + 2 \Vert p_{1}-x\Vert \leq d_{M}(q,p_{1})+2\beta < 3\beta.
\end{eqnarray*}
Hence, we have the inclusion 
\[
C_{M}(p_{1},\ldots,p_{k}) \subset B^{\mathbb{R}^N}(\pi(x),3\beta) \cap M = \{\pi(x)\},
\]
where the last equality follows from ($\beta$-2). Hence we have
\[
\emptyset \neq C_{M}(p_{1},\ldots,p_{k}) \cap (T_{\pi(x)}(M))^{\perp} \subset
B^{\mathbb{R}^N}(\pi(x),3\beta) \cap M = \{\pi(x)\},
\]
which implies $\{\pi(x)\} = C_{M}(p_{1},\ldots,p_{k}) \cap (T_{\pi(x)}M)^{\perp}$, as desired.

\end{proof}

\textbf{Proof of Theorem~\ref{onedim}.}~~
(1) First we prove that the projection $\pi_\beta$ is a homotopy equivalence.  We take a point $x\in \sh(\Ri_{\beta}(M))$ and take points $p_{1},\ldots, p_{k}$ of $M$ such that
\[
x \in \conv(\{p_{1},\ldots,p_{k}\})~~\mbox{and}~
d_{M}(p_{i},p_{j}) < \beta~(i,j=1,\ldots,k).
\]
By Lemma \ref{intermediate}, we see $\pi_{\beta}(x) \in C_{M}(p_{1},\ldots,p_{k})$. Since $\diam_{M}(C_{M}(p_{1},\ldots,p_{k})) < \beta$, we have
\[
\diam_{M}(\{\pi(x),p_{1},\ldots,p_{k}\}) < \beta,
\]
which implies $(1-t) x + t \pi_{\beta}(x) \in \sh(\Ri_\beta(M))$ for each $t\in [0,1]$. The rest of the proof proceeds exactly in the same way as that of Proposition~\ref{prop:SbetaM}.

Let $T_{\beta}:\Ri_{\beta}(M) \to M$ be the homotopy equivalence defined in (\ref{hausmannmap}).  Repeating the proof of Proposition \ref{projectionhausmann} and using ($\beta$-3), we see that 
\[
T_{\beta} \simeq \pi_{\beta}\circ p_{\beta}.
\]
Thus, we see that $p_\beta$ is also a homotopy equivalence.

 \bigskip
\noindent

(2)  Again we first prove that the projection $\pi_{\beta}^{S}$ is a homotopy equivalence.  Let $S$ be a finite subset of $M$ which is $\beta/2$-dense. We take an arc-length parametrization $\gamma:[0,\ell]\to M$ of $M$ with $\gamma(0) = \gamma(\ell)$ and enumerate $S$ as 
\[
S = \{\gamma(t_{1}),\gamma(t_{2}),\ldots,\gamma(t_{n}) \}
\]
where $t_{1}<t_{2}< \cdots <t_{n}$. Since $S$ is $\beta/2$-dense in $M$, we see 
\begin{equation}\label{rect-1}
d_{M}(\gamma(t_{i}),\gamma(t_{i+1})) < \beta
\end{equation}
for each $i=1,\ldots, k-1$. 

Let $L$ be the rectlinear curve in $\mathbb{R}^N$ 
defined by
\[
L = \cup_{i=1}^{n}\overline{\gamma(t_{i})\gamma(t_{i+1})} \cup \overline{\gamma(t_{n})\gamma(t_{1})}
\]
%obtained by consecutively joining $\gamma(t_{i})$ and $\gamma(t_{i+1})$ with the line segment $\overline{\gamma(t_{i})\gamma(t_{i+1})}$. 

From (\ref{rect}) below, we see $L\subset \sh(\Ri_\beta(S))$.
We define a map $g_{\beta}^{S}:M\to L(t_{1},\ldots,t_{n})$ as follows: 
for each $t\in [t_{i},t_{i+1}]$, the points $\gamma(t_{i})$ and $\gamma(t_{i+1})$ belong to distinct half-spaces determined by the $(n-1)$-hyperplane 
$(T_{\gamma(t)}(M))^{\perp}$, and hence the line segment $\overline{\gamma(t_{i})\gamma(t_{i+1})}$ intersects with $(T_{\gamma(t)}(M))^{\perp}$ in exactly one-point.  We define $g_{\beta}^{S}(\gamma(t))$ as 
the unique point $\overline{\gamma(t_{i})\gamma(t_{i+1})} \cap (T_{\gamma(t)}(M))^{\perp}$: 
\[
\{ g_{\beta}^{S}(\gamma(t))\} = \overline{\gamma(t_{i})\gamma(t_{i+1})} \cap (T_{\gamma(t)}(M))^{\perp}.
\]

We verify that $\pi_{\beta}\circ g_{\beta}^{S} = \operatorname{id}_{M}$ and $g_{\beta}^{S} \circ \pi_{\beta} \simeq \operatorname{id}_{\sh(\Ri_{\beta}(M))}$.

Since $\pi_{\beta}^{S}(T_{\gamma(t)}(M))^{\perp} = \gamma(t)$, we see
$\pi_{\beta}^{S} (g_{\beta}^{S}(\gamma(t)))  = \gamma(t)$.  Hence 
$\pi_{\beta}^{S}\circ g_{\beta}^{S} = \operatorname{id}_M$.
For each $x\in \sh(\Ri_\beta(S))$, we take points $p_{1},\ldots,p_{k}$ of $S$ such that $x \in \conv(\{p_{1},\ldots, p_{k}\})$ and $\diam_{M}(\{p_{1},\ldots,p_{k}\}) < \beta$.  Assume that $\pi_{\beta}^{S}(x)$ is written as $\pi_{\beta}^{S}(x) = \gamma(t),~ t \in [t_{i},t_{i+1}]$. By Lemma \ref{intermediate}, we have $\pi_{\beta}^{S}(x) = \gamma(t) \in C_{M}(p_{1},\ldots,p_{k})$. 
Since there are no points of $S$ in $\gamma(t_{i},t_{i+1})$, we have the inclusion 
$\gamma([t_{i},t_{i+1}]) \subset C_{M}(p_{1},\ldots,p_{k})$.  In particular
$\diam_{M}(\{ \gamma(t_{i}), \gamma(t_{i+1}), p_{1},\ldots, p_{k}\}) <\beta$ 
and hence 
\[
\conv(\{ \gamma(t_{i}), \gamma(t_{i+1}), p_{1},\ldots, p_{k}\}) \subset \sh(\Ri_{\beta}(S)).
\]
Thus we have the map $\sh(\Ri_\beta(S))\times [0,1] \to \sh(\Ri_\beta(S))$ given by
\[
(x,t) \mapsto (1-t)x + t g_{\beta}^{S}(\pi_{\beta}^{S}(x)),~~t \in [0,1]
\]
is a homotopy between $\operatorname{id}_{\sh(\Ri_\beta(S))}$ and $g_{\beta}^{S}\circ \pi_{\beta}^{S}$. 

Applying Proposition \ref{projectionhausmann} to a sample set $S\subset M$ which is $\beta/2$-dense in $M$ and using the above together with ($\beta$-3), we see that 
the shadow projection map $p_{\beta}^{S}:\Ri_{\beta}(S) \to \sh(\Ri_{\beta}(S))$ is a homotopy equivalence.

These prove (2).

\bigskip
\noindent

(3).  First observe that $\pi^{-1}(p) \subset (T_{p}(M))^{\perp}$ for each $p\in M$. Take $\tau>0$ and $S \subset N_{\tau}(M)$ as in the hypothesis.  Again we
take an arc-length parametrization $\gamma:[0,\ell]\to M$ of $M$ with $\gamma(0) = \gamma(\ell)$ and enumerate the finite subset $\pi(S)$ of $M$ as 
\[
\pi(S) = \{\gamma(t_{1}),\gamma(t_{2}),\ldots,\gamma(t_{n}) \}
\]
where $t_{1}<t_{2}< \cdots <t_{n}$. Since $\pi(S)$ is $\zeta$-dense in $M$, we see 
\begin{equation}\label{rect}
d_{M}(\gamma(t_{i}),\gamma(t_{i+1})) < 2\zeta
\end{equation}
for each $i=1,\ldots, k-1$. Let $S_{i} = S\cap \pi^{-1}(\gamma(t_{i}))$ and pick a point $x_{i} \in S_i$ for each $i=1, \ldots, n$.  Let $K$ be the rectlinear curve in $\mathbb{R}^N$ defined by
\[
K = \cup_{i=1}^{n-1} \overline{x_{i}x_{i+1}} \cup \overline{x_{n}x_{1}}.
\]
It follows directly that $K$ is a simple closed curve.  Also we see
\begin{eqnarray*}
\Vert x_{i}-x_{i+1}\Vert &\leq& \Vert x_{i} -\gamma(t_{i})\Vert + \Vert \gamma(t_{i})-\gamma(t_{i+1})\Vert + \Vert \gamma(t_{i+1}) - x_{i+1}\Vert\\
&<& 2\tau + 2\zeta < \beta,
\end{eqnarray*}
where the the last inequality follows from the hypothesis.
The above implies that $K \subset \sh(\Ri_{\beta}(S))$.
 Also since
there are no points of $\pi(S)$ in $\gamma(t_{i},t_{i+1})$, we see 
\[
\pi(\overline{x_{i}x_{i+1}}) = \gamma([t_{i},t_{i+1}])
\]
and 
\[
\overline{x_{i}x_{i+1}} \cap \pi^{-1}(\gamma(t)) ~\mbox{is a singleton}.
\]
Let $f:M\to K$ be the map given by $f(\gamma(t_{i})) = x_{i},~i=1,\ldots, n$ and
\[
\{f(\gamma(t)) = \overline{x_{i}x_{i+1}}\cap \pi^{-1}(\gamma(t)).
\]
for each $t\in (t_{i},t_{i+1})$.  
%We show that $f$ is ambiently isotopic to $\operatorname{id}_{\mathbb R^{N}}$.
Observe that $\displaystyle N_{\tau}(M) = \cup_{p \in M}D_{p}$, where $D_{\gamma(t)}$ is an $(N-1)$ dimensional disk contained in $(T_{\gamma(t)}(M))^{\perp}$ such that
$\gamma(t), f(t) \in D_{\gamma(t)}$ for each $t\in [0,\ell]$.  For $p \in M$, we have an isotopy $H_{p}:D_{p}\times [0,1] \to D_{p}$ such that
\begin{itemize}
\item[(1)] $H_{\gamma(t)}(x,0) = x$ for each $x\in D_\gamma(t)$ and $H_{\gamma(t)}(\gamma(t),1) = f(t)$ 
for each $t \in [0,\ell]$.
\item[(2)] $H_{p}(x,s) = x$~ for each $x\in \partial D_{p}$ and $s\in [0,1]$.
\item[(3)] The map $H:N_{\tau}(M)\times [0,1] \to N_{\tau}(M)$ defined by
\[
H(x,s) = H_{\pi(x)}(x,s),~~(x,s) \in N_{\tau}(M)\times [0,1]
\]
is an isotopy.
 \end{itemize}
The isotopy $H$ above naturally extends to that on $\mathbb{R}^N$ which fixes each point outside of $N_{\tau}(M)$.  Thus $K$ and $M$ are equivalently embedded in $\mathbb R^N$.

This proves (3).

\section{Discussions and Open Problems}
In this study, we have successfully proved the most intuitive limit theorems regarding Vietoris--Rips complexes and their Euclidean shadows around well-behaved Euclidean subsets $M$.
In the spirit of finite reconstruction of Euclidean shapes, we also show that the limits indeed stabilize, in case $M$ is a smooth, simple closed curve (Theorem~\ref{onedim}).
At the same time, our investigation raises numerous open questions and suggests new directions for exploration.
We list some of them below.

\begin{enumerate}[(1)]
\item
Is the space $\sh(\Ri_\beta(M)$ an ANR?  For a comact metric subspace $M$ of $\mathbb R^N$, we have the equality
\[
M = \cap_{\beta>0}\sh(\Ri_{\beta}(M)).
\]
Hence if the above question has an affirmative answer, then it would be helpful to investigate shape theoretic property of (not necessarily an ANR) $M$ by means of spaces $\sh(\Ri_{\beta}(M))$.

\item How does $\eta(M)$ as defined in Section~\ref{sec:recon} relate to the reach of a closed curve $M$?

\item In view of Theorem~\ref{onedim}, the following (complete) finite reconstruction question seems natural, but remains unanswered.
\begin{conjecture}
Let $M$ be a smooth, simple closed curve in $\mathbb{R}^N$. For any sufficiently small $\beta>0$ and for any finite set $S$ sufficiently close to $M$ in the sense of Hausdorff distance, we have a homotopy equivalence $\sh(\Ri^{\R^N}_{\beta}(S)) \simeq M$.
\end{conjecture}

\item To what extent can the results of Theorem~\ref{onedim} be generalized to higher-dimensional manifolds?

%{\color{blue}In essence, $\pi(\conv(\sigma))\subset\conv_M(\sigma)$ for any $\sigma\subset M$ with diameter small enough. 

%{Q1: Can this property be abstracted for higher-dimensional manifolds?
%k: I am not sure.  For higher dimensional manifolds, I could not find the way to controll the image $\pi(\conv(\sigma))$. 

%Q: It looks like for non-negative (sectional) curvature model spaces, this property is true at small scales.
%k: I am not sure. Question 1): Is the difference between the orthogonal projection and the exponential map well controlled for non-positive curvature point?  Question 2) Can a closed submanifold $M$ of $\mathbb R^N$ have non-positive curvature at every point? Standard embeddings of a torus and a closed surface of genus>=2 do not satisfy this. }

%{\color{blue}[[In view of the questions at the beginning of the proof, $\diam_M(\conv_M(\sigma))<\beta$; see \cite[Theorem 5.6]{Virk2021-lc}]]}
\end{enumerate}

\bibliography{main}
\bibliographystyle{plain}

\end{document}